\documentclass{article}
   \usepackage{amsmath,amsfonts,amsthm,amssymb,amscd,graphicx}
\theoremstyle{plain}
\newtheorem{thm}{Theorem}[section]
\newtheorem{prop}[thm]{Proposition}
\newtheorem{lemma}[thm]{Lemma}
\newtheorem{cor}[thm]{Corollary}

\theoremstyle{definition}
\newtheorem{defn}[thm]{Definition}

\theoremstyle{remark}
\newtheorem{rmk}[thm]{Remark}

\newcommand{\vol}{\ensuremath{\mathsf{vol}}}

\newcommand{\G}{\ensuremath{\operatorname{G_2}}}
\newcommand{\SP}{\ensuremath{\operatorname{Spin}(7)}}
\newcommand{\U}{\ensuremath{\operatorname{U}(1)}}

\newcommand{\Gs}{\ensuremath{\operatorname{G_2}}{-structure}}
\newcommand{\SPs}{\ensuremath{\operatorname{Spin}(7)}{-structure}}

\newcommand{\Z}{\ensuremath{\mathbb Z}}
\newcommand{\R}{\ensuremath{\mathbb R}}
\newcommand{\C}{\ensuremath{\mathbb C}}

\newcommand{\ph}{\ensuremath{\varphi}}
\newcommand{\ps}{\ensuremath{\psi}}

\newcommand{\wtwf}{\ensuremath{\Omega^2_{14}}}
\newcommand{\wtho}{\ensuremath{\Omega^3_1}}
\newcommand{\wths}{\ensuremath{\Omega^3_7}}
\newcommand{\wtht}{\ensuremath{\Omega^3_{27}}}

\newcommand{\st}{\ensuremath{\ast}}
\newcommand{\hk}{\mathbin{\! \hbox{\vrule height0.3pt width5pt 
depth 0.2pt \vrule height5pt width0.4pt depth 0.2pt}}}
\newcommand{\stph}{\ensuremath{\ast \varphi}}
\newcommand{\lieg}{\ensuremath{\operatorname{\mathfrak {g}_2}}}
\newcommand{\tr}{\ensuremath{\operatorname{Tr}}}

\newcommand{\ddx}[1]{\ensuremath{\frac{\del}{\del\:\!\! x^{#1}}}}
\newcommand{\dx}[1]{\ensuremath{d\:\!\! x^{#1}}}
\newcommand{\nab}[1]{\ensuremath{\nabla_{\! \! #1 \,}}}
\newcommand{\oph}{\ensuremath{\varphi_{\!\text{o}}}}

\newcommand{\del}{\ensuremath{\partial}}
\newcommand{\delbar}{\ensuremath{\bar \partial}}
\newcommand{\ddt}{\ensuremath{\frac{\partial}{\partial t}}}

\newcommand{\stp}{\ensuremath{\ast_{\varphi}}}
\newcommand{\HtwR}{\ensuremath{H^2 (M, \mathbb R)}}
\newcommand{\HtR}{\ensuremath{H^3 (M, \mathbb R)}}
\newcommand{\HfR}{\ensuremath{H^4 (M, \mathbb R)}}

\newcommand{\HtZ}{\ensuremath{H^3 (M, \mathbb Z)}}
\newcommand{\HfZ}{\ensuremath{H^4 (M, \mathbb Z)}}
\newcommand{\HtS}{\ensuremath{H^3 (M, S^1)}}
\newcommand{\HfS}{\ensuremath{H^4 (M, S^1)}}
\newcommand{\llangle}{\ensuremath{\langle \! \langle}}
\newcommand{\rrangle}{\ensuremath{\rangle \! \rangle}}
\newcommand{\nl}{\ensuremath{| \! |}}
\newcommand{\Hol}{\ensuremath{\mathrm{Hol}}}
\newcommand{\GJ}{\ensuremath{\mathcal G_{\! \mathcal J}}}

\numberwithin{equation}{section}
\numberwithin{table}{section}
\numberwithin{figure}{section}

\begin{document}

\title{Hodge Theory for \G-manifolds: \\
Intermediate Jacobians and Abel-Jacobi maps}

\author{Spiro Karigiannis \\ Mathematical Institute \\ University of
Oxford \\ \ \\ Naichung Conan Leung \\ Institute of Mathematical
Sciences and \\ Department of Mathematics \\ Chinese University of Hong Kong}

\maketitle

\begin{abstract}
We study the moduli space $\mathcal M$ of torsion-free \G-structures
on a fixed compact manifold $M^7$, and define its associated {\em
universal intermediate Jacobian} $\mathcal J$. We define the Yukawa
coupling and relate it to a natural pseudo-K\"ahler structure on
$\mathcal J$.

We consider natural Chern-Simons type functionals, whose critical
points give associative and coassociative cycles (calibrated
submanifolds coupled with Yang-Mills connections), and also deformed
Donaldson-Thomas connections. We show that the moduli spaces of these
structures can be isotropically immersed in $\mathcal J$ by means of
\G-analogues of {\em Abel-Jacobi} maps.
\end{abstract}

\section{Introduction} \label{introsec}
Compact manifolds with \G\ holonomy play the same role in M-theory as
Calabi-Yau threefolds do in string theory, both being Ricci-flat
and admitting a parallel spinor. In fact, there are analogues of the
mirror symmetry phenomenon~\cite{A, GYZ} that are expected to hold in
the context of \G-manifolds. These ideas are still not well understood
mathematically.

There are various special geometric structures that can be associated
to a \G-manifold and there are relationships between these structures.
Specifically we wish to consider analogues in \G-geometry of {\em
intermediate Jacobians} and {\em Abel-Jacobi maps} which are
familiar in algebraic geometry.

The moduli space of Calabi-Yau $3$-folds is known to admit a {\em
special Kahler structure}. Equivalently, this means that the {\em
universal intermediate Jacobian} $\mathcal J$ admits a {\em hyperK\"ahler
structure}. The images in $\mathcal J$ of the Abel-Jacobi maps are
{\em Lagrangian submanifolds}. Some references for these facts
are~\cite{BG, ClII, DM2, F, Katz}. In this paper we prove
analogues of these statements for \G-manifolds. We now describe
the organization of the paper, and mention the main results of
each section.

In Section~\ref{reviewsec} we briefly review some well-known facts
about \G-manifolds, which we will need later, and discuss notation.

In Section~\ref{specialsec} we begin with the description of the
moduli space $\mathcal M$ of torsion-free \G-structures on a fixed
\G-manifold $M$, which has been studied by Joyce~\cite{J1, J4} and
Hitchin~\cite{Hi4}. We define a natural pseudo-Riemannian metric on
$\mathcal M$ which is the Hessian of a superpotential function $f$.
Finally we define the Yukawa coupling $\mathcal Y$ on $\mathcal M$,
a fully symmetric cubic tensor, and relate this to the superpotential
function $f$.

Using these data, in Section~\ref{jacobianssec} we define intermediate
\G-Jacobians and the universal intermediate \G-Jacobian $\mathcal J$ of
the moduli space and show that it admits a natural pseudo-K\"ahler
structure, with K\"ahler potential (essentially) given by $f$. Also,
the projection $\mathcal J \to \mathcal M$ is a Lagrangian fibration,
and the associated cubic form to this fibration (as described
in~\cite{DM2}) is exactly the Yukawa coupling $\mathcal Y$.

Section~\ref{abeljacobisec} discusses how special geometric structures
associated to a \G-manifold $M$ can be viewed as critical points of
certain natural functionals $\Phi_{k, \ph}$ of Chern-Simons
type. Specifically we prove that the calibrated submanifolds of $M$,
together with special connections are such critical points, as are
deformed Donaldson-Thomas connections. We study each situation
individually.

Additionally, we show that the moduli spaces of such structures can be
immersed (via Abel-Jacobi type mappings) into $\mathcal J$
isotropically with respect to the appropriate symplectic structure on
$\mathcal J$. Therefore, these images are Lagrangian subspaces
whenever they are half-dimensional. We also discuss a {\em topological
number} $\Psi_{k, \ph}$ on the moduli spaces of these structures,
and relate it to the critical points of $\Phi_{k, \ph}$.

Finally, in Section~\ref{conclusionsec} we briefly consider the
generalization of these results to the case when the cohomology of
the \G-manifold $M$ has torsion, which involves {\em gerbes}. We
also discuss questions for future study.

{\bf Acknowledgements.} The first author would like to thank Bobby Acharya for useful discussions concerning the physics of \G-manifolds, and in particular for informing him about their use of $F = -\log(f)$ as the superpotential function. The first author is also greatly indebted to Dominic Joyce for pointing out some problems with an initial draft of this article, and for useful discussions with Christopher Lin and John Loftin. The authors also thank the referee for useful comments which improved an earlier version of this article. The research of the second author is partially supported by a research RGC grant from the Hong Kong government. 

\section{Review of \G-structures} \label{reviewsec}
In this section we briefly review some facts about \G-structures. Some references for \G-structures are~\cite{Br},~\cite{J4},~\cite{K1},~\cite{K3},~\cite{L3}, and~\cite{Sa}. In addition, the first examples of compact irreducible $\G$~manifolds are described in~\cite{J1} and~\cite{J4}.

Let $M^7$ be a closed manifold with a \G-structure. The \G-structure
is given by a positive $3$-form $\ph$, which induces an orientation
and a Riemannian metric $g_{\ph}$ in a non-linear way via the formula
\begin{equation*}
(X \hk \ph) \wedge (Y \hk \ph) \wedge \ph = - 6 \, g_{\ph}(X,Y) \, \vol_{\ph}
\end{equation*}
where $\vol_{\ph}$ is the volume form corresponding to the orientation
and the metric $g_{\ph}$. Hence there is an induced Hodge star
operator $\stp$ and the associated dual $4$-form $\ps = \stp \ph$.
The $3$-form has constant pointwise norm $|\ph|^2 = 7$ with respect to
$g_{\ph}$.

The \G-structure is called {\em torsion-free} if $\ph$ is parallel
with respect to its induced metric $g_{\ph}$. In this case the
Riemannian holonomy is contained in \G, and $M^7$ is called a \G-{\em
manifold}.  It is well known that a \G-structure $\ph$ is torsion-free
if and only if $\ph$ is both closed and co-closed (with respect to
$g_{\ph}$.)

The space of forms $\Omega^*$ on a manifold with \G-structure
decomposes into \G-representations, and this decomposition
descends to the cohomology in the torsion-free case. Specifically, the
cohomology breaks up as
\begin{eqnarray*}
\HtwR & = & H^2_7 \oplus H^2_{14} \\ \HtR & = & H^3_1 \oplus H^3_7 \oplus
H^3_{27} \\ \HfR & = & H^4_1 \oplus H^4_7 \oplus H^4_{27} \\
H^5 (M, \mathbb R) & = & H^5_7 \oplus H^5_{14}
\end{eqnarray*}
We define $b^k_l = \dim(H^k_l)$. Then $b^3_1 = b^4_1 = 1$, and $b^k_7
= b_1(M)$, for $k = 2, 3, 4, 5$. The actual holonomy group $\Hol(g_{\ph})$
of $(M, g_{\ph})$ is determined by the topology of $M$. For example, the holonomy is
exactly $\G$ if and only if $\pi_1(M)$ is finite, in which case $b_1(M) = 0$ and all the $b^k_7 = 0$.
Such an $M^7$ is called {\em irreducible}.

A result which we will need repeatedly is the following.

\begin{lemma} \label{ddtpslemma}
Let $\ph_t = \ph_0 + t \eta$ be a one-parameter family of \G-structures
for small $t$. Then we have
\begin{equation} \label{ddtpseq}
\left. \ddt \right|_{t=0} ( \st_{\ph_t} \ph_t) = \frac{4}{3} \st_{\ph_0}
\pi_1 (\eta) + \st_{\ph_0} \pi_7 (\eta) - \st_{\ph_0} \pi_{27} (\eta)
\end{equation}
where $\st_{\ph_0}$ is the Hodge star induced from $\ph_0$, and $\pi_k$ is the
orthogonal projection onto $\Omega^3_k$ (defined with respect to $\ph_0$.)
\end{lemma}
\begin{proof}
This is first mentioned in~\cite{J1}, and explicit proofs can be found
in~\cite{Hi4} and~\cite{K3}.
\end{proof}

It follows from Lemma~\ref{ddtpslemma} that if $\ps_t = \ps_0 + t \theta$
is a one-parameter family of postive $4$-forms, then
\begin{equation*}
\left. \ddt \right|_{t=0} ( \st_{\ps_t} \ps_t) = \frac{3}{4} \st_{\ph_0}
\pi_1 (\theta) + \st_{\ph_0} \pi_7 (\theta) - \st_{\ph_0} \pi_{27} (\theta)
\end{equation*}
which motivates the following definition.

\begin{defn} \label{stardefn}
Let $\ph_0$ be a fixed torsion-free \G-structure. We define the map
$\star_{\ph_0} : \Omega^k \to \Omega^{7-k}$ for $k = 3,4$ by
\begin{eqnarray} \label{stardefneq}
\text{ for } \eta \in \Omega^3; \qquad \star_{\ph_0} (\eta) & = &
\frac{4}{3} \st_{\ph_0} \pi_1 (\eta) + \st_{\ph_0} \pi_7(\eta) - \st_{\ph_0}
\pi_{27}(\eta) \\ \text{ for } \theta \in \Omega^4; \qquad \star_{\ph_0}
(\theta) & = & \frac{3}{4} \st_{\ph_0} \pi_1(\theta) + \st_{\ph_0} \pi_7 (\theta)
- \st_{\ph_0} \pi_{27}(\theta)
\end{eqnarray} 
Notice that $\star_{\ph_0}$ agrees (up to a constant) with
$\st_{\ph_0}$ on each $\Omega^k_l$, but the constants are different
on different components. Also note that $\star_{\ph_0}^2 = 1$.
\end{defn}

Finally we make some remarks about notation. We use $\langle \alpha,
\beta \rangle$ and $| \alpha |^2$ to denote the pointwise inner
product and norm on forms induced from $g_{\ph}$. We use $\llangle
\alpha, \beta \rrangle = \int_M \langle \alpha, \beta \rangle
\vol_{\ph}$ and ${\nl \alpha \nl}^2 = \llangle \alpha, \alpha
\rrangle$ to denote the corresponding $L^2$ inner product.

As $M$ is always taken to be compact, we use Hodge theory throughout,
so we will often identify a cohomology class $[\alpha]$
with its unique harmonic representative $\alpha$. We use $g = g_{\ph}$
to denote the Riemannian metric on $M^7$ associated to the
\G-structure $\ph$, and reserve $\mathcal G$ for the
metric on the moduli space $\mathcal M$ of \G-manifolds, and
$\GJ$ for the pseudo-K\"ahler metric on the universal intermediate Jacobian
$\mathcal J$. The summation convention is used throughout, and the dimension
of the moduli space  $\mathcal M$ is $b_3 = n + 1$ with indices running from
$0$ to $n$.

\section{Special Geometry of the \G-moduli space} \label{specialsec}

\subsection{The moduli space of \G-manifolds}
\label{modulisec}

We denote by $\mathcal M$ the moduli space of torsion-free
\G-structures on $M$. Explicitly, $\mathcal M$ is defined as
\begin{equation*}
\mathcal M = \left\{ \ph \in \Omega^3_+ (M): \nab{\ph} \ph = 0 \right\} /
\mathrm{Diff}_0 (M)
\end{equation*}
where $\Omega^3_+(M)$ denotes the space of \G-structures on $M$ (also
called the {\em positive} $3$-forms), and $\mathrm{Diff}_0 (M)$ is the
space of diffeomorphisms of $M$ isotopic to the identity.
This should more properly be called the {\em Teichmuller
space} of \G-structures on $M$.

\begin{rmk} \label{teichmullerrmk}
The quotient of the torsion-free \G-structures by the {\em full}
diffeomorphism group of $M$ (which is equivalent to the quotient of
$\mathcal M$ by the mapping class group of $M$) is in general an {\em
orbifold}.  As most of our discussions are local, we could equally
consider this moduli space, at least where it is smooth.
\end{rmk}

Dominic Joyce proved in~\cite{J1, J4} that $\mathcal M$ is a smooth
manifold of dimension $b_3 (M)$. Note that on a compact \G-manifold, $b_3 
\geq 1$, since $\ph$ is a non-trivial harmonic $3$-form. Therefore the moduli
space is always at least one-dimensional, and those moduli correspond to 
constant scalings of $\ph$, and hence scalings of the total volume
of $M$. In~\cite{J1, J4}, Joyce also proved the following local result.

\begin{thm}[Joyce~\cite{J1, J4}] \label{modulispacethm}
The period map
\begin{eqnarray*}
P & : & \mathcal M \rightarrow \HtR \\
P \left( \ph \right) & = & \left[ \ph \right]
\end{eqnarray*}
is a local diffeomorphism. Here $[ \ph ]$ denotes the cohomology class
of $\ph$. Hence, given $\oph \in \mathcal M$, there exists
$\epsilon > 0$ such that
\begin{equation*}
\mathcal U_{\oph} = \left\{ \ph \in \Omega^3_+ (M) :
{\nl \ph - \oph \nl}_{g_{\oph}} < \epsilon, \nab{\ph} \ph = 0
\right\} / \mathrm{Diff}_0 (M)
\end{equation*}
is diffeomorphic to an open set in the vector space $\HtR$.
\end{thm}

This result says $\mathcal M$ has a natural {\em affine} structure
and a {\em flat connection} $\nab{}$, and $P$ is an affine map. To see
this, let $\eta_0, \ldots, \eta_n$ be a basis of $\HtR$, where $n + 1 =
b_3 = b^3_7 + b^3_{27} + 1$. Let $x^i$ be coordinates on $\mathcal U_{\ph_0}$ with
respect to this basis. Thus $\eta = x^i \eta_i$ describes a point in
$\mathcal U_{\ph_0}$ and $\eta_i = \ddx{i}$.  Since $\mathcal
U_{\ph_0}$ is an open set in a vector space, it admits a flat
connection $\nab{} = d$, and because $\nab{} (\dx{i}) = d(\dx{i}) =
0$, we say $x^i$ are {\em flat} coordinates. We can thus cover
$\mathcal M$ by such flat charts.

On two overlapping flat charts, the coordinate systems $\tilde x^i$
and $x^i$ are affinely related: $x^i = P^i_j \tilde x^j + Q^i$, where
$P^i_j$ and $Q^i$ are constants.  Thus $\frac{\partial}{\partial
  \tilde x^j} = P^i_j \ddx{i}$, and the transition functions are
constant. Hence the flat connection $\nab{}$ is a well-defined
connection on $\mathcal M$, since if $\nab{} = d + A$ with $A = 0$ in
one flat chart, then $\tilde A = P^{-1} A P + P^{-1} dP = 0$ in all
overlapping flat charts. Such structures also arise in~\cite{Hi1, Hi2,
L2, LYZ}, for example.

There is a natural {\em pseudo-Riemannian} metric $\mathcal G$ on the
moduli space $\mathcal M$ first defined by Hitchin~\cite{Hi4} which we
now describe.

\begin{defn} \label{potentialdefn}
We define a smooth function $f: \mathcal M \to \R$ as follows:
\begin{equation} \label{potentialdefneq}
f(\ph) = \frac{3}{7} \int_M (\ph \wedge \stp \ph) =
3 \int_M \vol_{\ph}
\end{equation}
where $\ph$ is a point in $\mathcal M$. Thus, up to a constant,
$f([\ph])$ is the total volume of $M$ with respect to the metric and
orientation induced by $\ph$. We call $f$ the {\em superpotential function}
for $\mathcal M$, for reasons that will soon be clear.
\end{defn}

\begin{rmk} \label{constantrmk}
Our definition differs from Hitchin's by a constant factor, which we
choose for later convenience.
\end{rmk}

In~\cite{Hi4}, Hitchin shows that when restricted to closed
\G-structures in a fixed cohomology class, the torsion-free
\G-structures are precisely those which are critical points of $f$. We
will say more about this in Section~\ref{DTsec}.

\begin{thm}[Hitchin~\cite{Hi4}] \label{hessianmetricthm}
In a flat coordinate chart $(\mathcal U_{\ph_0}, x^0, \ldots, x^n)$,
the Hessian $f_{ij} = \frac{\partial^2 f}{\partial x^i \partial x^j}$
defines a pseudo-Riemannian metric $\mathcal G$ on $\mathcal M$ by $\mathcal G_{ij} =
f_{ij}$. In the case when $\Hol(g_{\ph}) = \G$, the metric $\mathcal G_{ij}$
is {\em Lorentzian}.
\end{thm}
\begin{proof}
Let $\ph$ be a point in $\mathcal M$. We want to differentiate $f$
in the $\eta_j$ direction. We identify the cohomology class
$\eta_j$ with its unique harmonic
representative (with respect to the metric $g_{\ph}$ induced from $\ph$.)
Now by Lemma~\ref{ddtpslemma}, we have
\begin{eqnarray*}
\frac{\partial f}{\partial x^j} & = & \frac{3}{7} \int_M \left( \ddx{j} \ph \right)
\wedge \stp \ph + \frac{3}{7} \int_M \ph \wedge \left( \ddx{j} \stp \! \ph \right) \\
& = & \frac{3}{7} \int_M \eta_j \wedge \stp \ph
+ \frac{3}{7} \int_M \ph \wedge \left( \frac{4}{3} \stp \pi_1
(\eta_j) + \stp \pi_7 (\eta_j) - \stp \pi_{27} (\eta_j)
\right)\\ & = & \frac{3}{7} \int_M \pi_1 (\eta_j) \wedge \stp \ph +
\frac{4}{7} \int_M \ph \wedge \stp \pi_1 (\eta_j) \\ & = & \int_M
\pi_1 (\eta_j) \wedge \stp \ph = \int_M \eta_j \wedge \stp \ph
\end{eqnarray*}
where we have used the fact that the splitting of $\Omega^3$ is orthogonal
with respect to $g_{\ph}$.
To compute the second derivative, let $\ph(t)$ be a one-parameter family of
torsion-free \G-structures with $\ph(0) = \ph$, satisfying
$\ddt \ph(t) = \eta_i (t)$, where $\eta_i(t)$ is the unique harmonic
representative of the cohomology class $\eta_i$ with respect to the metric $g_t$
induced from $\ph(t)$. Then $\eta_j (t) = \eta_j (0) + d \beta(t)$ for some
smooth family of $2$-forms $\beta(t)$. Hence $\ddt \eta_j(t) = d \beta'(t)$ is
exact.

Writing $\eta_j (0) = \eta_j$, we compute the second derivative at $\ph$ as
\begin{eqnarray*}
& & \frac{\partial^2 f}{\partial x^i \partial x^j} = \int_M \left( \left. \ddt
\right|_{t=0} \eta_j(t) \right) \wedge \st_{\ph} \ph + \int_M \eta_j \wedge \left(
\left. \ddt \right|_{t=0} \st_{\ph(t)} \ph(t) \right) \\
& = & \int_M d\beta'(0) \wedge \st_{\ph} \ph + 
\int_M \eta_j \wedge \left( \frac{4}{3} \stp \pi_1 (\eta_i)
+ \stp \pi_7 (\eta_i) - \stp \pi_{27} (\eta_i) \right) \\
& = & \int_M \left( \frac{4}{3} \pi_1 (\eta_i) \wedge 
\stp \pi_1 (\eta_j) + \pi_7 (\eta_i) \wedge 
\stp \pi_7 (\eta_j) - \pi_{27} (\eta_i) \wedge 
\stp \pi_{27} (\eta_j) \right) \\ & = & \frac{4}{3}
\llangle \pi_1 (\eta_i), \pi_1 (\eta_j) \rrangle
+ \llangle \pi_7 (\eta_i), \pi_7 (\eta_j) \rrangle
- \llangle \pi_{27} (\eta_i), \pi_{27} (\eta_j) \rrangle
\end{eqnarray*}
where the first term in the second line above vanishes by Stokes' theorem
since $\stp \ph$ is closed and $M$ is compact.

The Laplace and Green's operators commute with the projections when
$\ph$ is torsion-free, so each $\pi_k (\eta_j)$ is
harmonic. Therefore we see that $f_{ij}$ is a pseudo-Riemannian metric
of signature $(b^3_1 + b^3_7, b^3_{27}) = (1 + b_1, b^3_{27})$.

When $M$ is irreducible (so the holonomy is exactly \G),
we know $b_1 = 0$. Since $H^1 \cong H^3_7$, there are no harmonic
$\wths$-forms, and thus in this case
\begin{equation} \label{hessianmetriceq}
f_{ij} = \frac{4}{3} \llangle \pi_1 (\eta_i), \pi_1 (\eta_j)
\rrangle - \llangle \pi_{27} (\eta_i), \pi_{27} (\eta_j) \rrangle
\end{equation}
which is a Lorentzian metric of signature $(1, b^3_{27})$.
\end{proof}

\begin{rmk} \label{metricwelldefinedrmk}
We defined this metric using
a particular flat coordinate chart in a neighbourhood of each point.
However, this metric $\mathcal G$ is well-defined globally on $\mathcal M$.
This is because with respect to any other overlapping flat coordinate
chart $(\tilde x^0, \ldots \tilde x^n)$, the two sets of coordinates
are affinely related: $x^i = P^i_j \tilde x^j + Q^i$. Therefore
$\frac{\partial}{\partial \tilde x^k} = P^i_k \ddx{i}$, and also
$\frac{\partial^2 f}{\partial \tilde x^k \partial \tilde x^l} =
P^i_k P^j_l \frac{\partial^2 f}{\partial x^i \partial x^j}$. This is
exactly the statement that the metric $\mathcal G$ is well-defined.
\end{rmk}

\begin{rmk} \label{physicsrmk}
In the physics literature, a slightly different notion of
superpotential function is used: they define $F = -\log(f)$ and use
the Hessian of $F$ to define a metric on $\mathcal M$. The advantage
of this definition is that the metric $F_{ij}$ is actually {\em
Riemannian} when $M$ is irreducible, as opposed to
Lorentzian. However, with this choice of potential function and
metric, other results fail to hold. See
Remarks~\ref{physics2rmk},~\ref{physics3rmk}, and~\ref{physics4rmk}
for more details. It is for these reasons that we
prefer to use $f$ and its associated Lorentzian metric.
\end{rmk}

We close this section by noting that we can write the metric $\mathcal G$ on
$\mathcal M$ more concisely using the operator $\star_{\ph}$ defined
in~\eqref{stardefneq} as
\begin{equation} \label{starmetriceq}
\mathcal G(\eta_1, \eta_2) = \int_M \eta_1 \wedge \star_{\ph} \eta_2
\end{equation}
for $\eta_i \in T_{\ph} \mathcal M = \HtR$. We stress again that
$\star_{\ph}$ is not the same as $\st_{\ph}$, and thus $\mathcal G$ is not the
usual $L^2$ metric on forms. This metric $\mathcal G$ is more
natural than the usual $L^2$ metric, because the fact that it is of
Hessian type allows us to construct a pseudo-K\"ahler structure on the
universal intermediate Jacobian in Section~\ref{universalsec}.

\subsection{The Yukawa Coupling} \label{yukawasec}

In this section we define the {\em Yukawa coupling} $\mathcal Y$,
a symmetric cubic tensor on the \G-moduli space $\mathcal M$. We relate this
tensor to the superpotential function $f$ and the metric $\mathcal G$ on
$\mathcal M$. We will need the following identity, which can be found
in Lemma A.12 of~\cite{K3}:
\begin{equation} \label{g2identitieseq}
\ph_{ijk} \ph_{abc} g^{kc} = g_{ia} g_{jb} - g_{ib} g_{ja} - \ps_{ijab}
\end{equation}
We recall some facts about the space $\Omega^3 = \wtho \oplus \wths
\oplus \wtht$ of $3$-forms on a \G-manifold. We adopt the notation of~\cite{K3}.
The elements $\eta \in \Omega^3_1 \oplus \Omega^3_{27}$ are in one-to-one correspondence
with symmetric $2$-tensors $h_{ij} \in S^2 (T^*)$. The correspondence is given by
\begin{equation} \label{isomeq}
\eta_{ijk} = h_{il} g^{lm} \ph_{mjk} + h_{jl} g^{lm} \ph_{imk} + h_{kl} g^{lm} \ph_{ijm}
\end{equation}
where $\wtho$ corresponds to multiples of the metric $g$ and $\wtht$
corresponds to the {\em traceless} symmetric tensors. Note that the
$3$-form $\ph$ itself corresponds to the symmetric tensor $\frac{1}{3}
g_{ij}$.

From now on we assume that $M$ is irreducible, so $\HtR = H^3_1 \oplus
H^3_{27}$. Each cohomology class has a unique harmonic representative, so for a fixed
torsion-free \G-structure $\ph$ with metric $g_{\ph}$, each class
$\eta_k \in \HtR$ corresponds to a symmetric $2$-tensor $h_k$.

\begin{defn} \label{yukawadefn}
The {\em Yukawa coupling} $\mathcal Y$ is a fully symmetric cubic tensor on the moduli space
$\mathcal M$, defined as follows. Let $\eta_i$, $\eta_j$, $\eta_k$ be elements
of $T_{\ph} \mathcal M \cong \HtR$, with associated symmetric tensors $h_i$, $h_j$, $h_k$,
with respect to $g_{\ph}$. Then we define
\begin{equation} \label{yukawadefneq}
\mathcal Y (\eta_i, \eta_j, \eta_k) = \int_M (h_i)^{a\alpha} (h_j)^{b\beta} (h_k)^{c\gamma}
\phi_{abc} \phi_{\alpha \beta \gamma} \vol
\end{equation}
where $g$ and $\vol$ are with respect to $\ph$, and $(h_i)^{a\alpha}$
means $(h_i)_{b\beta} g^{ba} g^{\beta \alpha}$, raising indices with
$g$. It is clear that $\mathcal Y$ is fully symmetric in its arguments.
\end{defn}

In Section~\ref{universalsec}, we will see that the Yukawa coupling
$\mathcal Y$ is essentially the natural cubic form associated to a
Lagrangian fibration.

Fix $\ph$ in $\mathcal M$. Let $\eta_0, \eta_1, \ldots, \eta_n$ be a basis
for $T_{\ph} \mathcal M \cong \HtR$, where $\eta_0 = \ph$ and $\eta_i \in H^3_{27}$ for
$i \neq 0$.

\begin{prop} \label{yukawafirstprop}
The following relations between the Yukawa coupling
$\mathcal Y$, the superpotential function $f$, and the Hessian metric $\mathcal G_{ij}
= f_{ij}$ hold:
\begin{eqnarray} \label{Y000eq}
\mathcal Y (\ph, \ph, \ph) & = & \frac{14}{27} \, f(\ph) \\
\label{Y00ieq} \mathcal Y (\ph, \ph, \eta_i) & = & 0 \\
\label{Y0ijeq} \mathcal Y (\ph, \eta_i, \eta_j) & = & \frac{1}{6} \, \mathcal G_{ij} (\ph)
\end{eqnarray}
where $i, j = 0, \ldots, n$, and $i \neq 0$ in~\eqref{Y00ieq}.
\end{prop}
\begin{proof}
First we show that~\eqref{Y000eq} and~\eqref{Y00ieq} follow from~\eqref{Y0ijeq}.
Equation~\eqref{Y00ieq} is automatic since $\mathcal G_{0i} = 0$ for $i \neq 0$
by~\eqref{hessianmetriceq}. Also~\eqref{hessianmetriceq} shows $\mathcal G_{00} = 
\frac{4}{3} \int \ph \wedge \st \ph = \frac{28}{9} f (\ph)$.
Then~\eqref{Y0ijeq} says $\mathcal Y (\ph, \ph, \ph) = \frac{1}{6} \mathcal G_{00}
= \frac{14}{27} f(\ph)$.

To establish~\eqref{Y0ijeq}, consider $h_k = \frac{1}{3} g$ in~\eqref{yukawadefneq} and
use~\eqref{g2identitieseq}. We obtain:
\begin{eqnarray} \nonumber
\mathcal Y (\ph, \eta_i, \eta_j) & = & \frac{1}{3} \int_M (h_i)^{a\alpha}
(h_j)^{b \beta} \ph_{abc} \ph_{\alpha \beta \gamma} g^{c \gamma} \vol \\ \nonumber
& = & \frac{1}{3} \int_M (h_i)^{a \alpha} (h_j)^{b \beta} \left( g_{a \alpha} g_{b \beta}
 - g_{a \beta} g_{b \alpha} - \ps_{ab\alpha \beta} \right) \vol \\
\label{yukawafirstproptempeq} & = & \frac{1}{3} \int_M \left( \tr(h_i) \tr(h_j) -
\tr(h_i h_j) \right) \vol
\end{eqnarray}
where $(h_i h_j)_{ab} = (h_i)_{al} g^{lm} (h_j)_{mb}$ denotes matrix multiplication.
The third term vanished by the symmetry of the $h$'s and the skew-symmetry of $\ps$.
Let $h_i = \lambda g + h^0_i$ and $h_j = \mu g + h^0_j$ where $h^0_i$ and $h^0_j$
are the trace-free parts. Substituting these into~\eqref{yukawafirstproptempeq} gives
\begin{equation} \label{yukawafirstproptempeq2}
\mathcal Y(\ph, \eta_i, \eta_j) = \frac{1}{3} \int_M \left( 42 \, \lambda \mu -
\tr(h^0_i h^0_j) \right) \vol
\end{equation}
In Proposition 2.15 of~\cite{K3}, it is shown that for $\eta_i, \eta_j \in \wtho \oplus
\wtht$, we have
\begin{eqnarray} \nonumber
\int_M (\eta_i \wedge \st \eta_j) \vol & = & \int_M \left( \tr(h_i) \tr(h_j) +
2 \tr(h_i h_j) \right) \vol \\ \label{L2metriceq} & = & \int_M \left( 63 \, \lambda \mu
+ 2 \tr(h^0_i h^0_j) \right) \vol
\end{eqnarray}
By~\eqref{hessianmetriceq}, the metric $\mathcal G_{ij} = f_{ij}$ is given by
\begin{equation*}
\mathcal G_{ij} = \int_M \eta_i \wedge \st \bar \eta_j \vol
\end{equation*}
where $\bar \eta_j$ corresponds to the symmetric tensor $\bar  h_j = \frac{4}{3} \mu g
- h^0_j$. Substituting $\mu \to \frac{4}{3} \mu$ and $h^0_j \to - h^0_j$
in~\eqref{L2metriceq} shows
\begin{equation} \label{yukawafirstproptempeq3}
\mathcal G_{ij} = \int_M \left( 84 \, \lambda \mu - 2 \tr(h^0_i h^0_j ) \right) \vol
\end{equation}
Comparing~\eqref{yukawafirstproptempeq2} and~\eqref{yukawafirstproptempeq3}, shows
$\mathcal Y (\ph, \eta_i, \eta_j) = \frac{1}{6} \, \mathcal G_{ij}$.
\end{proof}

From the above proposition, the natural question which arises is
whether $\mathcal Y(\eta_i, \eta_j, \eta_k)$ is related to
$f_{ijk}$. The main theorem of this section is the following, which
shows that this is indeed the case.

\begin{thm} \label{yukawathm}
Let $\eta_i, \eta_j, \eta_k$ be in $T_{\ph} \mathcal M \cong \HtR$. Then
\begin{equation} \label{Yijkeq}
\mathcal Y (\eta_i, \eta_j, \eta_k) = \frac{1}{2} \, f_{ijk} = \frac{1}{2} \,
\frac{\partial^3 f}{\partial x^i \partial x^j \partial x^k}
\end{equation}
for $i,j,k = 0, \ldots, n$. 
\end{thm}

Before we can prove this theorem, we need a couple of preliminary
results. The reason is that it is difficult to directly differentiate
$f_{ij} = \int_M \eta_i \wedge \star_{\ph} \eta_j$ because of the
complicated nature of the operator $\star_{\ph}$.  We circumvent this
difficulty by considering the auxiliary function $F = -\log(f)$
instead, because it will turn out that the third derivative of $F$ is
easily computable.

\begin{lemma} \label{yukawatemplemma1}
Let $F = - \log(f)$. Then the Hessian
$F_{ij} = \frac{\partial^2 F}{\partial x^i \partial x^j}$
is given by
\begin{equation} \label{yukawatemplemma1eq}
F_{ij} = \frac{1}{f} \int_M \eta_i \wedge \stp \eta_j = \frac{1}{f} \llangle \eta_i, 
\eta_j \rrangle
\end{equation}
when $M$ is irreducible.
\end{lemma}
\begin{proof}
In this proof, subscripts on $f$ or $F$ always denote partial
differentiation. We evaluate all derivatives at a fixed point $\ph$ in
$\mathcal M$. As before fix a basis $\eta_0, \ldots, \eta_n$ of
$\HtR$ so that $\eta_0 = \ph$. From the proof of Theorem~\ref{hessianmetricthm},
and our choice of basis, we see that at the point $\ph$ in $\mathcal M$,
\begin{equation} \label{fitempeq}
f_0 = \frac{7}{3} f \qquad \qquad  f_i = 0 \quad \text{for } i \neq 0 
\end{equation}
and also that
\begin{equation} \label{fijtempeq}
f_{00} = \frac{28}{9} f \qquad f_{i0} = 0 \qquad f_{ij} = -\int \langle \eta_i,
\eta_j \rangle \vol \quad \text{for } i,j \neq 0
\end{equation}
From $F = -\log(f)$ we have $F_i = -f^{-1} f_i$ and
\begin{equation} \label{Fijeq}
F_{ij} = \frac{1}{f^2} f_i f_j - \frac{1}{f} f_{ij}
\end{equation}
Substituting the above expressions, we see that at the point $\ph$ in $\mathcal
M$, we have
\begin{equation*}
F_{00} = \frac{7}{3} = \frac{1}{f} \int \ph \wedge \stp \ph \qquad F_{i0} = 0
\qquad F_{ij} = + \frac{1}{f} \int \eta_i \wedge \stp \eta_j \quad
\text{for } i,j \neq 0
\end{equation*}
and hence in all cases we get
\begin{equation*}
F_{ij} = \frac{1}{f} \int_M \eta_i \wedge \stp \eta_j 
= \frac{1}{f} \llangle \eta_i, \eta_j \rrangle \qquad i,j = 0, \ldots, n
\end{equation*}
as claimed.
\end{proof}

\begin{lemma} \label{yukawatemplemma2}
Let $\eta_1$, $\eta_2$, and $\eta_3$ be $3$-forms in $\wtho \oplus \wtht$, with associated
symmetric tensors $h_1$, $h_2$, and $h_3$, respectively. Then
\begin{multline} \label{yukawatemplemma2eq}
\int_M (\eta_1)_{ijk} (\eta_2)_{abc} (h_3)^{ia} g^{jb} g^{kc} \vol \, \, \, = \, \, 2 \, \mathcal Y(\eta_1, \eta_2, \eta_3) \\ + 2 \int_M \left( \tr(h_1) \tr(h_2 h_3) + \tr(h_2) \tr(h_3 h_1) + \tr(h_3) \tr(h_1 h_2) \right) \vol
\end{multline}
\end{lemma}
\begin{proof}
We use~\eqref{isomeq} and exploit the symmetry of $h_1$, $h_2$, $h_3$,
$g$, and the skew-symmetry of $\ph$ to compute:
\begin{eqnarray*}
(\eta_1)_{ijk} (\eta_2)_{abc} (h_3)^{ia} g^{jb} g^{kc} & = & (h_1)_{il} g^{lm} \ph_{mjk}
(\eta_2)_{abc} (h_3)^{ia} g^{jb} g^{kc} \\ & & {} + 2 \, 
(h_1)_{jl} g^{lm} \ph_{imk} (\eta_2)_{abc} (h_3)^{ia} g^{jb} g^{kc}
\end{eqnarray*}
and then
\begin{eqnarray*}
(\eta_1)_{ijk} (\eta_2)_{abc} (h_3)^{ia} g^{jb} g^{kc} & = & (h_1)_{il} (h_2)_{a\alpha}
g^{lm} g^{\alpha \beta} \ph_{mjk} \ph_{\beta bc} (h_3)^{ia} g^{jb} g^{kc} \\ & &  {} + 
2 \, (h_1)_{il} (h_2)_{b\alpha} g^{lm} g^{\alpha \beta} \ph_{mjk} \ph_{a \beta c}
(h_3)^{ia} g^{jb} g^{kc} \\ & & {} + 2 \, (h_1)_{jl} (h_2)_{a\alpha} g^{lm} g^{\alpha \beta}
\ph_{imk} \ph_{\beta bc} (h_3)^{ia} g^{jb} g^{kc} \\ & & {} + 2 \,
(h_1)_{jl} (h_2)_{b\alpha} g^{lm} g^{\alpha \beta} \ph_{imk} \ph_{a \beta c}
(h_3)^{ia} g^{jb} g^{kc} \\ & & {} + 2 \, (h_1)_{jl} (h_2)_{c\alpha} g^{lm} g^{\alpha \beta}
\ph_{imk} \ph_{a b \beta} (h_3)^{ia} g^{jb} g^{kc}
\end{eqnarray*}
We use the identity~\eqref{g2identitieseq} repeatedly, and simplify.
The right hand side becomes
\begin{eqnarray*}
& & 6 \, \tr(h_1 h_2 h_3) + 2  \left( \tr(h_2) \tr(h_3 h_1) - \tr(h_2 h_3 h_1) \right) \\
& & {}+ 2 \left( \tr(h_1) \tr(h_2 h_3) - \tr(h_1 h_2 h_3) \right) + 
2 \left( \tr(h_3) \tr(h_1 h_2) - \tr(h_3 h_1 h_2) \right) \\ & & {} +
2 (h_1)^{ia} (h_2)^{jb} (h_3)^{kc} \ph_{ijk} \ph_{abc}
\end{eqnarray*}
which, upon further simplification and integration over $M$, is
exactly~\eqref{yukawatemplemma2eq}.
\end{proof}

\begin{proof}[Proof of Theorem~\ref{yukawathm}]
We begin by differentiating~\eqref{Fijeq} with respect to $x^k$ and
rearranging to obtain
\begin{equation} \label{fijkeq}
f_{ijk} = -f F_{ijk} - \frac{2}{f^2} f_i f_j f_k + \frac{1}{f} f_i f_{jk}
+ \frac{1}{f} f_j f_{ki} + \frac{1}{f} f_k f_{ij}
\end{equation}
We need to compute $F_{ijk}$. Using Lemma~\ref{yukawatemplemma1}, this is
\begin{eqnarray} \nonumber
F_{ijk} & = & -\frac{1}{f^2} f_k \int_M \eta_i \wedge \stp \eta_j + \frac{1}{f}
\int_M \left( \ddx{k} \eta_i \right) \wedge \stp \eta_j \\ \label{yukawathmtempeq} & & {}+
\frac{1}{f} \int_M \left( \ddx{k} \eta_j \right) \wedge \stp \eta_i + \frac{1}{f}
\int_M \eta_i \wedge \left( \ddx{j} \stp \right) \eta_j
\end{eqnarray}
As in the proof of Theorem~\ref{hessianmetricthm}, we have $\ddx{k} \eta_i = d\beta'_i$
is exact. Since the $\eta_i$'s are harmonic with respect to $g_{\ph}$, we have that
$\stp \eta_i$ is closed, and hence by Stokes' theorem the second and third terms above
both vanish. The fourth term of~\eqref{yukawathmtempeq} is
\begin{equation} \label{yukawathmtempeq2}
\frac{1}{f} \int_M \ddx{k} \left( g_{\ph} (\eta_i, \eta_j) \right) \vol + \frac{1}{f}
\int_M g_{\ph} (\eta_i, \eta_j) \ddx{k} \vol
\end{equation}
where we regard $\eta_i$ and $\eta_j$ as constant. In local coordinates, we have
\begin{equation*}
g_{\ph} (\eta_i, \eta_j) = \frac{1}{6} \, (\eta_i)_{abc} (\eta_j)_{\alpha\beta\gamma}
g^{a\alpha} g^{b\beta} g^{c\gamma}
\end{equation*}
We are differentiating in the $x^k$ direction, which means that we are considering
a one-parameter family $\ph(t)$ of torsion-free \G-structures, such that
$\ddt \ph(t) = \eta_k (t)$, which at $t=0$ corresponds to a symmetric $2$-tensor $h_k$.
From Corollary 3.3 of~\cite{K3}, such an infinitesmal variation induces
\begin{equation*}
\ddx{k} g^{ab} = -2 \, (h_k)^{ab} \qquad \qquad \ddx{k} \vol = \tr(h_k) \, \vol
\end{equation*}
as the variations of $g^{ab}$ and $\vol$, respectively. Substituting these results
and simplifying expression~\eqref{yukawathmtempeq2}, equation~\eqref{yukawathmtempeq}
becomes
\begin{eqnarray} \nonumber
F_{ijk} & = & -\frac{1}{f^2} f_k \int_M \eta_i \wedge \stp \eta_j 
+ \frac{1}{f} \int_M \tr(h_k) \, \eta_i \wedge \stp \eta_j \\ \label{yukawathmtempeq3}
& & {} +  \frac{1}{f} \int_M \left( - (\eta_i)_{abc} (\eta_j)_{\alpha\beta\gamma}
(h_k)^{a\alpha} g^{b\beta} g^{c\gamma} \right) \vol 
\end{eqnarray}
Now there are two cases. If $k \neq 0$, then $\tr(h_k) = 0$ and $f_k = 0$
from~\eqref{fitempeq}. Whereas if $k=0$, then $h_0 = \frac{1}{3} g$, so $\tr(h_0) = 
\frac{7}{3}$, and $f_0 = \frac{7}{3} f$ by~\eqref{fitempeq}. Thus in all cases the
combination of the first two terms of~\eqref{yukawathmtempeq3} vanishes. Now applying
Lemma~\ref{yukawatemplemma2}, we finally obtain:
\begin{eqnarray*}
-f F_{ijk} & = & 2 \, \mathcal Y(\eta_1, \eta_2, \eta_3) \\ & & {}+
2 \int_M \left( \tr(h_i) \tr(h_j h_k) +
\tr(h_j) \tr(h_k h_i) + \tr(h_k) \tr(h_i h_j) \right) \vol
\end{eqnarray*}
Substituting this into~\eqref{fijkeq} gives
\begin{eqnarray} \label{fijkfinaleq}
f_{ijk} & = &  2 \, \mathcal Y(\eta_1, \eta_2, \eta_3)
- \frac{2}{f^2} f_i f_j f_k + \frac{1}{f} f_i f_{jk}
+ \frac{1}{f} f_j f_{ki} + \frac{1}{f} f_k f_{ij} \\ \nonumber & & {}+
2 \int_M \left( \tr(h_i) \tr(h_j h_k) +
\tr(h_j) \tr(h_k h_i) + \tr(h_k) \tr(h_i h_j) \right) \vol
\end{eqnarray}

{\em Case 1: $i=j=k=0$.} Then since $h_0 = \frac{1}{3} g$, we have
$\tr(h_0) = \frac{7}{3}$ and $\tr(h_0 h_0) = \frac{7}{9}$, and $f_0 = \frac{7}{3} f$ and
$f_{00} = \frac{28}{9} f$ by~\eqref{fitempeq} and~\eqref{fijtempeq}.
Thus~\eqref{fijkfinaleq} is
\begin{eqnarray*}
f_{000} & = & 2 \, \mathcal Y(\eta_0, \eta_0, \eta_0) - \frac{2}{f^2} {\left(
\frac{7}{3} f \right)}^3  + \frac{3}{f} \left( \frac{7}{3} f \right) \left(
\frac{28}{9} f \right) + 6 \left( \frac{7}{3} \right) \left( \frac{7}{9} \right)
\int_M \vol \\ & = & 2 \, \mathcal Y(\eta_0, \eta_0, \eta_0) - \frac{686}{27} f
+ \frac{196}{9} f + \frac{294}{81} f = 2 \, \mathcal Y(\eta_0, \eta_0, \eta_0)
\end{eqnarray*}
where we have used $f = 3 \int \vol$.

{\em Case 2: $j=k=0$, $i \neq 0$.} We have $\tr(h_i) = 0$ and $\tr(h_i h_0) = 0$.
Also $f_i = 0$ and $f_{i0} = 0$ by~\eqref{fitempeq} and~\eqref{fijtempeq}.
Thus all the extra terms vanish and~\eqref{fijkfinaleq} in this case says
\begin{equation*}
f_{i00} = 2 \, \mathcal Y(\eta_i, \eta_0, \eta_0)
\end{equation*}

{\em Case 3: $k=0$, $i,j \neq 0$.} In this case the only non-vanishing terms on the
right hand side of~\eqref{fijkfinaleq}, besides the $\mathcal Y$ term, are the last
terms on each line. This time we get
\begin{equation*}
f_{ij0} =  2 \, \mathcal Y(\eta_i, \eta_j, \eta_0) + \frac{7}{3} f_{ij} + \frac{14}{3}
\int_M \tr(h_i h_j) \, \vol 
\end{equation*}
Since $h_i$ and $h_j$ are traceless, equation~\eqref{yukawafirstproptempeq3} shows
that the last two terms cancel, and hence
\begin{equation*}
f_{ij0} =  2 \, \mathcal Y(\eta_i, \eta_j, \eta_0)
\end{equation*}

{\em Case 4: $i,j,k \neq 0$.} As in Case 2, all the extra terms of~\eqref{fijkfinaleq}
vanish, leaving
\begin{equation*}
f_{ijk} =  2 \, \mathcal Y(\eta_i, \eta_j, \eta_k)
\end{equation*}
thus finally completing the proof of Theorem~\ref{yukawathm}.
\end{proof}

\begin{rmk} \label{physics2rmk}
Note that Proposition~\ref{yukawafirstprop} and Theorem~\ref{yukawathm}
give more reasons to prefer the superpotential function $f$ that we are
considering, rather than $F = -\log(f)$, since it is clear from their proofs
that both $F_{ij}$ and $F_{ijk}$ are not as simple and natural as $f_{ij}$
and $f_{ijk}$. See also Remarks~\ref{physicsrmk},~\ref{physics3rmk},
and~\ref{physics4rmk}.
\end{rmk}

\section{Intermediate Jacobians in \G-geometry} \label{jacobianssec}

We begin with the following observation. All of the special geometric
structures that exist on a \G-manifold $M$ are defined using both the
parallel $3$-form $\ph$ and its associated Hodge-dual parallel
$4$-form $\ps = \stph$. These include associative and coassociative
submanifolds of $M$, as well as Donaldson-Thomas bundles on $M$.  A
discussion of these structures is in Section~\ref{abeljacobisec}. For
this reason, we consider the {\em pair} $(\ph, \stph) \in \HtR \oplus
\HfR$ as representing the \Gs\ corresponding to $\ph$. In this context
we have the following result of Joyce.

\begin{prop}[Joyce~\cite{J4}] \label{lagprop}
By Poincar\'e duality, $\HfR \cong \HtR^*$, so the space
$\HtR \oplus \HfR$ has a natural symplectic structure. Define the
subset $\mathcal L$ of $\HtR \oplus \HfR$ by
\begin{equation*}
\mathcal L = \{ ([\ph], [\stph]): \ph \text{ is a torsion-free \Gs} \}
\end{equation*}
Then $\mathcal L$ is a {\em Lagrangian submanifold} of $\HtR \oplus \HfR$ with
respect to the natural symplectic structure.
\end{prop}
\begin{proof}
In the proof of Theorem~\ref{hessianmetricthm}, we showed that
\begin{equation*}
\left. \frac{\partial f}{\partial t} \right|_{\ph} =
\int_M \frac{\partial \ph}{\partial t} \wedge \stph
\end{equation*}
which says that $df|_{[\ph]} = [\stph]$ under the isomorphism
$\HtR^* = \HfR$. Therefore $\mathcal L$ is locally of the form
$([\ph], df|_{[\ph]})$, the graph of the gradient of a real function $f$.
Hence $\mathcal L$ is Lagrangian.
\end{proof}

\begin{rmk} \label{lagrmk}
It is clear that Theorem~\ref{modulispacethm} can be generalized to show
that the maps $([\ph], [\stp \ph]) \mapsto [\ph] \in \HtR$ and
$([\ph], [\stp \ph]) \mapsto [\stp \ph] \in \HfR$ are local diffeomorphisms.
In fact, by projection onto the first factor, we see $\mathcal L$ is
diffeomorphic to $\mathcal M$, the \G-moduli space.
\end{rmk}

\subsection{Intermediate \G-Jacobians} \label{intermediatesec}

We first explain the motivation for introducing (intermediate)
Jacobians.  For an algebraic curve $C$, we associate to it
the Jacobian $\mathcal J \left( C \right) = H^{1,0} \left( C
\right) \backslash H^1 \left( C, \mathbb C \right) / H^1 \left(
C,\mathbb Z \right)$, which is an Abelian variety. It determines the curve $C$ itself.  Similarly,
for any complex threefold $X$, we define its intermediate Jacobian to
be the Abelian variety
\begin{equation*}
\mathcal J \left( X \right) = \left( H^{3,0} \left( X \right) + H^{2,1} \left( X \right) \right) \backslash H^3
\left( X,\mathbb C \right) / H^3 \left( X,\mathbb Z \right),
\end{equation*}
and it captures much of the algebraic structure of $X$, especially when $X$
is a Calabi-Yau threefold. Now we will define a similar object, a flat
torus which encodes the information of a torsion-free \G-structure on $M^7$.

From now on we assume that the cohomology of $M$ has no torsion.
If not, we need to consider {\em gerbes}. See Section~\ref{conclusionsec}
for a discussion of the general case. We use the notation
\begin{equation} \label{toruscohomeq}
H^k(M, S^1) = H^k(M, \R) / H^k(M, \Z)
\end{equation}
which is topologically a torus of dimension $b_k(M)$. Now consider the splitting
\begin{equation*}
\HtR \oplus \HfR = H' \oplus H''
\end{equation*}
where we define
\begin{eqnarray*}
H' & = & (1 + \star_{\ph}) \HtR = \{ (C, \star_{\ph} C ): C \in \HtR \} \\
H'' & = & (1 - \star_{\ph}) \HtR  = \{ (C, -\star_{\ph} \! C ): C \in \HtR \} = (H')^{\perp}
\end{eqnarray*}
where $\star_{\ph}$ is the operator of Definition~\ref{stardefn}.
Recall we are implicitly identifying a cohomology class $[C]$ with
its harmonic representative.
Clearly we have $H'' \cong \HtR \oplus \HfR / H'$.
Now let $\Lambda = (1 - \star_{\ph}) \HtZ$ be the image of the lattice $\HtZ$ under the
projection onto $H''$. 

\begin{defn} \label{intermediatejacobiandefn}
We define the intermediate \G-Jacobian $\mathcal J_{\ph}$ by
\begin{eqnarray} \nonumber
\mathcal J_{\ph} & = & (1 + \star_{\ph}) \HtR \backslash \HtR \oplus \HfR / (1 - \star_{\ph})
\HtZ \\ \label{intermediatejacobiandefneq} & = &  H'' / \Lambda
\end{eqnarray}
for each $\ph \in \mathcal M$.
\end{defn}

By projection onto the second factor, we have the isomorphism
\begin{equation} \label{intermediatejacobiandefneq2}
\mathcal J_{\ph} \cong \HfR / \star_{\ph} \!\! \HtZ
\end{equation}
Equation~\eqref{intermediatejacobiandefneq2} says that we can think of $\mathcal
J_{\ph}$ as a quotient of $\HfR$ by a varying lattice which detects
the \G-structure.

For a Calabi-Yau $3$-fold $X^6$ or a \G-manifold $M^7$, the complement
of the tangent space to the intermediate Jacobian in $H^3(X, \C)$
or $\HtR \oplus \HfR$, respectively, is the tangent space to the local
moduli space of deformations. In the Calabi-Yau case, the space
$H^{3,0} \oplus H^{2,1}$ is the infinitesmal moduli space of
deformations of the holomorphic volume form $\Omega$ (which
corresponds to deformations of complex structures and volume
scalings). In the \G\ case, this complement is the set
$H' = \{ (\eta, \star_{\ph} \eta): \eta \in \HtR \}$ which by
Lemma~\ref{ddtpslemma} and Remark~\ref{lagrmk} is exactly
the tangent space at $\ph$ to the moduli space of torsion-free
\G-structures.

We can define a similar object $\tilde{\mathcal J}_{\ph}$ by
\begin{eqnarray} \nonumber
\tilde{\mathcal J}_{\ph} & = & (1 + \star_{\ph}) \HfR \backslash \HfR
\oplus \HtR / (1 - \star_{\ph}) \HfZ \\ \label{intermediatejacobiandefneq3}
& \cong & \HtR / \star_{\ph} \!\! \HfZ
\end{eqnarray}

The two tori $\mathcal J_{\ph}$ and $\tilde{\mathcal J}_{\ph}$ are
dual to each other. We can use either one to define the universal
intermediate \G-Jacobian in Section~\ref{universalsec}. In the former
case the resulting space will be locally isomorphic to $T^* \mathcal
M$, and in the latter case to $T \mathcal M$. Both viewpoints will be
needed, depending on whether we are considering the associative or
coassociative moduli. This is discussed in Sections~\ref{instantonssec}
and~\ref{branessec}.

\subsection{The Universal Intermediate \G-Jacobian}
\label{universalsec}

We now consider all the intermediate \G-Jacobians $\mathcal J_{\ph}$
for all torsion-free \G-structures $\ph$ in $\mathcal M$.

\begin{defn} \label{universaldefn}
We define the {\em universal intermediate \G-Jacobian} $\mathcal J$ to
be the fibre bundle over $\mathcal M$ whose fibre over a point $\ph$
in $\mathcal M$ is the intermediate Jacobian $\mathcal J_{\ph}$
associated to $\ph$. That is,
\begin{equation*}
\mathcal J = \{ (\ph, \mathcal J_{\ph}); \ph \in \mathcal M \}
\end{equation*}
\end{defn}

\begin{rmk} \label{Cfieldrmk}
We also call $\mathcal J$ the moduli
space of torsion-free \G-structures, together with {\em C-fields}.
The elements of the fibre $\mathcal J_{\ph}$ over each point $\ph$ in
$\mathcal M$ are called C-fields with respect to $\ph$ in the
physics literature.
\end{rmk}

By the isomorphism~\eqref{intermediatejacobiandefneq2}, we see that
topologically, $\mathcal J$ is isomorphic to $\mathcal M \times \HfS$.
Locally, $\HfS$ is of course isomorphic to $\HfR$, and therefore
$\mathcal J$ is locally isomorphic to $\mathcal M \times \HfR \cong T^* \mathcal M$,
since $H^4 = (H^3)^{\ast}$ by Poincar\'e duality.

\begin{thm} \label{bigmodulispacekahlerthm}
The universal intermediate Jacobian $\mathcal J$ admits the structure of a
pseudo-K\"ahler manifold.
\end{thm}

We are going to give two proofs of
Theorem~\ref{bigmodulispacekahlerthm}, to illustrate clearly the
relationships between the various structures.

\begin{proof}[Proof 1]
Let $\mathcal U_{\oph}$ be a flat coordinate chart in $\mathcal M$
with flat coordinates $x^0, \ldots, x^n$. Then the canonical
coordinates $x^0, \ldots, x^n, y_0, \ldots, y_n$ on $T^* \mathcal
U_{\oph}$ descend to coordinates on the restriction of $\mathcal J$
to $\mathcal U_{\oph}$, where we use the same letter $y_i$ to denote
the angle coordinates on the fibre. Let
\begin{equation*}
\omega = dx^i \wedge dy_i
\end{equation*}
be the canonical symplectic form on $\mathcal J$ defined using
the local isomorphism with $T^* \mathcal M$. It is easy to see that
this is globally well-defined on $\mathcal J$.

Now we define new local coordinates
$x_0, \ldots , x_n$ on the base $\mathcal U_{\oph}$ by
\begin{equation} \label{legendreeq}
x_k = \frac{\partial f}{\partial x^k}
\end{equation}
where $f$ is the superpotential function. These are the new coordinates
which are used to define the Legendre transform of the function $f$.
To see that these coordinates are well defined, observe
that~\eqref{legendreeq} implies
\begin{equation} \label{legendreeq2}
dx_k = f_{kj} \dx{j}
\end{equation}
where $f_{kj} = \frac{\partial^2 f}{\partial x^k \partial x^j}$ is
invertible. Note that the right hand side is closed, since $f_{kjl} =
f_{klj}$, and thus it is exact in $\mathcal U_{\oph}$ (we can assume
that the flat coordinate charts are contractible.) Therefore the new
coordinates $x_k$ are well defined up to a constant, which we can
fix by demanding that the origin corresponds to $\oph$, the origin
with respect to the original coordinates. Now
\begin{equation*}
\frac{\partial}{\partial x_0}, \ldots, \frac{\partial}{\partial x_n},
\frac{\partial}{\partial y_0}, \ldots \frac{\partial}{\partial y_n}
\end{equation*}
is a local basis of coordinate vector fields on $\mathcal J$.
We define an endomorphism $J$ of the tangent bundle of
$\mathcal J$ by
\begin{equation*}
J \left( \frac{\partial}{\partial x_k} \right) =
\frac{\partial}{\partial y_k} \qquad \qquad 
J \left( \frac{\partial}{\partial y_k} \right) =
-\frac{\partial}{\partial x_k} \qquad \qquad k = 0, \ldots, n
\end{equation*}
It is a simple exercise to check that if $x^i = P^i_j \tilde x^j +
Q^j$ for two overlapping flat coordinate charts, then $d\tilde x_k
= P^i_k d x_i$ and $d\tilde y_k = P^i_k d y_i$, hence
$J$ is well-defined globally on $\mathcal J$. This was the reason for
considering the coordinate transformation from $x^i$ to $x_i$. This is
the Legendre transform corresponding to $f$, as described for example
in~\cite{L2}. It is clear that $J^2 = -1$ and that in fact $J$ is an
{\em integrable} complex structure on $\mathcal J$, since $z_k = x_k +
i y_k$ are local holomorphic coordinates. Now let us define a metric
$\GJ$ on $\mathcal J$ by the compatibility relation
\begin{equation} \label{kahlercompatibilityeq}
\GJ (X, Y) = \omega(X, JY)
\end{equation}
It follows easily from $J^2 = -1$ and~\eqref{kahlercompatibilityeq}
that $\GJ$ is $J$-invariant. From~\eqref{legendreeq2}
it is clear that
\begin{equation*}
\ddx{k} = f_{kj} \frac{\partial}{\partial x_j}
\end{equation*}
Let $f^{ij}$ be the inverse matrix of $f_{ij}$. With respect to the
coordinates $x^i, y_i$, the metric $\GJ$ is
\begin{eqnarray*}
\GJ \! \left( \ddx{i} , \ddx{j} \right) & = & \omega
\left( \ddx{i} , J \ddx{j} \right) = f_{jl} \omega
\left( \ddx{i} , \frac{\partial}{\partial y_l} \right) = f_{ji} \\
\GJ \! \left( \ddx{i} , \frac{\partial}{\partial y_j}\right)
& = & \omega \left( \ddx{i} , J \frac{\partial}{\partial y_j} \right)
= -f^{jl} \omega \left( \ddx{i} , \ddx{l} \right) = 0 \\
\GJ \! \left( \frac{\partial}{\partial y_i} ,
\frac{\partial}{\partial y_j} \right) & = & \omega
\left( \frac{\partial}{\partial y_i} , J \frac{\partial}{\partial y_j}
\right) = -f^{jl} \omega \left( \frac{\partial}{\partial y_i} ,
\ddx{l} \right) = f^{ji}
\end{eqnarray*}
which we can summarize as
\begin{equation} \label{kahlermetriceq}
\GJ = \mathcal G_{ij} \dx{i} \otimes \dx{j} + \mathcal G^{ij} dy_i
\otimes dy_j
\end{equation}
recalling that $f_{ij} = \mathcal G_{ij}$. Thus $(\omega, J, \GJ)$ is
a pseudo-K\"ahler structure on $\mathcal J$.
\end{proof}

\begin{proof}[Proof 2]
Alternatively, we could proceed as follows.
Theorem~\ref{hessianmetricthm} gives us the Riemannian metric
$\mathcal G_{ij}$ on the base $\mathcal M$ of $\mathcal J$. We use the flat
connection $\nab{}$ on $\mathcal M$ to lift the coordinate vector
fields $\eta_i = \ddx{i}$ to their {\em horizontal lifts} $\tilde
\eta_i = \ddx{i}$ (since $\nab{} = d + A$ with $A = 0$.) Now we use
the connection $\nab{}$ to lift the metric $\mathcal G$ from $\mathcal M$ to
a metric $\GJ$ on $\mathcal J$ by
isometrically identifying the horizontal space with the tangent
space to the base, and using the fibre metric on the cotangent
bundle induced from the metric $\mathcal G$ on $\mathcal M$, since $\mathcal J$
is locally isomorphic to $T^* \mathcal M$. Explicitly,
\begin{eqnarray*}
\GJ (\tilde \eta_i , \tilde \eta_j) & \equiv & \mathcal G \, (\eta_i, \eta_j)
= \mathcal G_{ij} \\ \GJ \! \left( \tilde \eta_i,
\frac{\partial}{\partial y_j}
\right) & \equiv & 0 \\ \GJ \! \left( \frac{\partial}{\partial y_i},
\frac{\partial}{\partial y_j} \right) & \equiv & \mathcal G
\left( \frac{\partial}{\partial y_i}, \frac{\partial}{\partial y_j}
\right) = \mathcal G^{ij}
\end{eqnarray*}
From~\eqref{kahlermetriceq} we see that this is the same metric $\GJ$
again. Let $\omega$ be the standard symplectic form on
$\mathcal J$ as before. Now we can
use~\eqref{kahlercompatibilityeq} to define $J$. It is easy to check
using~\eqref{kahlermetriceq} and~\eqref{kahlercompatibilityeq} that
this gives
\begin{equation}
J \left( \ddx{i} \right) = \mathcal G_{ij} \frac{\partial}{\partial y_j}
\qquad \qquad J \left( \frac{\partial}{\partial y_i} \right) =
- \mathcal G^{ij} \ddx{j}
\end{equation}
from which it follows that $J^2 = -1$. Now $\zeta_k = \ddx{k} - i
\mathcal G_{kj} \frac{\partial}{\partial y_j}$ is a basis of $(1,0)$ vector
fields. The Lie bracket of two $(1,0)$ vector fields of this type is
easily computed to be
\begin{equation*}
[ \zeta_k, \zeta_l ] = i \left( \frac{\partial \mathcal G_{kj}}{\partial x^l}
\frac{\partial}{\partial y_j} - \frac{\partial \mathcal G_{lj}}{\partial x^k}
\frac{\partial}{\partial y_j} \right) = 0
\end{equation*}
since $\mathcal G_{ij} = \frac{\partial^2 f}{\partial x^i \partial x^j}$ is a
Hessian. Therefore the $[1,0]$ vector fields are closed under the Lie
bracket, and thus by the Newlander-Nirenberg theorem, $J$ is
integrable. Hence $(\omega, J, \GJ)$ is a
pseudo-K\"ahler structure on $\mathcal J$.
\end{proof}

We pause now to describe $g_{\mathcal J}$, $\omega$, and $J$
invariantly. At a point $(\ph, D)$ in $\mathcal M \times
\HfS \cong \mathcal J$, the canonical symplectic form $\omega$ on
$\mathcal J$ is given by
\begin{equation} \label{standardsymplecticeq}
\omega( (\eta_1, \theta_1), (\eta_2, \theta_2) ) =
\int_M \eta_1 \wedge \theta_2 - \int_M \eta_2 \wedge
\theta_1
\end{equation}
where $\eta_i \in T_{\ph} \mathcal M = \HtR$ and $\theta_i \in
T_{D} (\HfS) \cong \HfR = (\HtR)^*$, by Poincar\'e duality. Also,
using~\eqref{stardefneq},~\eqref{starmetriceq}, and~\eqref{kahlermetriceq} we see
\begin{equation} \label{invariantmetriceq}
\GJ( (\eta_1, \theta_1), (\eta_2, \theta_2) ) =
\int_M \eta_1 \wedge \star_{\ph} \eta_2 + \int_M \theta_1 \wedge
\star_{\ph} \theta_2
\end{equation}
Now from~\eqref{kahlercompatibilityeq},~\eqref{standardsymplecticeq},
and~\eqref{invariantmetriceq} it follows that
\begin{equation} \label{invariantJeq}
J(\eta, \theta) = (- \star_{\ph} \theta, \star_{\ph} \eta)
\end{equation}
Since $\star_{\ph}^2 = 1$, this equation shows clearly that
$J^2 = -1$.

\begin{cor} \label{lagrangianfibrationcor}
The fibration $\pi : \mathcal{J} \rightarrow \mathcal{M}$ is a
Lagrangian fibration with a Lagrangian section. That is, the zero
section $\mathcal M$ and the fibres of the projection map $\pi$ are
Lagrangian submanifolds of $\mathcal J$.
\end{cor}
\begin{proof}
This is immediate from the fact that $\omega$ is the canonical
symplectic form on $\mathcal J$ given by the local isomorphism with
the cotangent bundle $T^* \mathcal M$.
\end{proof}

We also have the following further relationship between the superpotential
function $f$ and the pseudo-K\"ahler
structure on $\mathcal J$.

\begin{prop} \label{kahlerpotentialprop}
The Legendre transform $\hat f$ of $f$ is a K\"ahler potential for the
pseudo-K\"ahler structure on $\mathcal J$. That is,
\begin{equation*}
\omega = 2 i \del \delbar \hat f = 2 i \left(
\frac{\partial^2 \hat f}{\partial z_i \partial \bar z_j} \right) dz_i
\wedge d\bar z_j
\end{equation*}
Furthermore, $\hat f (\ph) = \frac{4}{3} f(\ph) = \frac{4}{7} \int_M
(\ph \wedge \stp \ph)$.
\end{prop}
\begin{proof}
The definition of the Legendre transform $\hat f$ of $f$ is
\begin{equation*}
\hat f (x_0, \ldots, x_n) = x_i x^i(x_0, \ldots, x_n) -
f(x^0(x_0, \ldots, x_n), \ldots, x^n(x^0, \ldots, x_n))
\end{equation*}
as a function of the $x_i$'s, where recall that $x_i = \frac{\partial
f}{\partial x^i}$. An easy calculation shows
\begin{equation*}
\frac{\partial^2 \hat f}{\partial x_i \partial x_j} = f^{ij}
\end{equation*}
where $f^{ij}$ is the inverse matrix of $f_{ij} = \frac{\partial^2
f}{\partial x^i \partial x^j}$. Since $\omega$ is $J$-invariant, it
is of type $(1,1)$, and we have
\begin{equation*}
\omega = \frac{i}{2} \mathcal G^{\alpha \beta} dz_{\alpha} \wedge d\bar z_{\beta}
=  \frac{i}{2} \frac{\partial^2 \hat f}{\partial x_{\alpha}
\partial x_{\beta}} dz_{\alpha} \wedge d\bar z_{\beta} = 2 i
\frac{\partial^2 \hat f}{\partial z_{\alpha} \partial \bar z_{\beta}}
dz_{\alpha} \wedge d\bar z_{\beta}
\end{equation*}
and hence $\omega = 2 i \del \delbar \hat f$. All that remains is to
establish that $\hat f = \frac{4}{3} f$. Let $\ph = x^i \eta_i$. Using
the computations in the proof of Theorem~\ref{hessianmetricthm}, we
see that
\begin{eqnarray*}
\hat f & = & x^i x_i - f = x^i \frac{\partial f}{\partial x^i} - f = x^i
\int_M \eta_i \wedge \stp \ph - \frac{3}{7} \int_M \ph \wedge \stp
\ph \\ & = & \int_M \ph \wedge \stp \ph - \frac{3}{7} \int_M
\ph \wedge \stp \ph = \frac{4}{7} \int_M \ph \wedge \stp \ph
\end{eqnarray*}
as claimed.
\end{proof}

\begin{rmk} \label{physics3rmk}
Using $F = -\log(f)$ as a superpotential, one can still obtain a K\"ahler
structure on $\mathcal J$. In fact this time it is really K\"ahler since the
Hessian metric $F_{ij}$ on $\mathcal M$ is Riemannian (if and only if
$M$ is irreducible.) However, the Legendre transform $\hat F$ is {\em not}
a constant multiple of $F$ in this case. See also
Remarks~\ref{physicsrmk},~\ref{physics2rmk}, and~\ref{physics4rmk}.
\end{rmk}

The pseudo-K\"ahler structure on $\mathcal J$ we have just described
arises from the pseudo-Riemannian metric $\mathcal G$ on $\mathcal M$ together
with the local isomorphism of $\mathcal J$ with $T^* \mathcal M$.
This structure will be used to study the moduli space of branes
(coassociative submanifolds) of $M$ in Section~\ref{branessec}.

Alternatively, we could consider the other definition of the
intermediate Jacobian $\tilde{\mathcal J}_{\ph}$ given
in~\eqref{intermediatejacobiandefneq3}, which gives us that the
universal intermediate Jacobian $\mathcal J$ is locally
isomorphic to $T \mathcal M$. We take $x^0, \ldots, x^n$ to be local
flat coordinates on a flat chart $\mathcal U_{\oph}$ as before, and
let $y^0, \ldots, y^n$ be the corresponding coordinates on the
fibre. This gives a local basis of coordinate vector fields
\begin{equation*}
\frac{\partial}{\partial x^0}, \ldots, \frac{\partial}{\partial x^n},
\frac{\partial}{\partial y^0}, \ldots \frac{\partial}{\partial y^n}
\end{equation*}
and we define an endomorphism
$\tilde J$ of the tangent bundle of $\mathcal J$ by
\begin{equation*}
\tilde J \left( \frac{\partial}{\partial x^k} \right) =
\frac{\partial}{\partial y^k} \qquad \qquad 
\tilde J \left( \frac{\partial}{\partial y^k} \right) =
-\frac{\partial}{\partial x^k} \qquad \qquad k = 0, \ldots, n
\end{equation*}
Then $\tilde J$ is an integrable complex structure on $\mathcal J$ with
local holomorphic coordinates $z^k = x^k + i y^k$.

As before we use the flat connection $\nab{}$ on $\mathcal M$ to
lift the metric $\mathcal G$ on $\mathcal M$ to a metric $\tilde \GJ$ on
$\mathcal J$ by isometrically identifying the
horizontal space with the tangent space to the base, and using the
fibre metric on the tangent bundle induced from the metric $\mathcal G$ on
$\mathcal M$, since $\mathcal J$ is locally isomorphic to
$T \mathcal M$. That is,
\begin{equation*}
\tilde \GJ = \mathcal G_{ij} \dx{i} \otimes \dx{j} + \mathcal G_{ij} dy^i
\otimes dy^j
\end{equation*}
It is clear that this metric $\tilde \GJ$ is $\tilde J$-invariant.
This time we use the compatibility relation to define the symplectic form
$\tilde \omega$ from $\tilde J$ and $\tilde \GJ$ by
\begin{equation*}
\tilde \omega(X, Y) = \tilde \GJ (\tilde J X, Y)
\end{equation*}
In coordinates, this becomes
\begin{equation*}
\tilde \omega = \mathcal G_{ij} \dx{i} \wedge dy^j
\end{equation*}
which is closed because $\frac{\partial \mathcal G_{ij}}{\partial x^l} =
\frac{\partial^3 f}{\partial x^i \partial x^j \partial x^l}$ is
symmetric in $i$ and $l$.  Equivalently we can write $\tilde \omega = dx_{j}
\wedge dy^j$ where $dx_j = \mathcal G_{ij} \dx{i}$ as before, which clearly shows
the closure of $\tilde \omega$.

The triple $(\tilde \omega, \tilde J, \tilde \GJ )$ is a
different pseudo-K\"ahler structure on $\mathcal J$. In this case it is
easy to check directly that the superpotential function $f$ is exactly a
K\"ahler potential for this structure:
\begin{equation*}
\tilde \omega = 2i \del \delbar f = 2i \left(
\frac{\partial^2 f}{\partial z^i
\partial \bar z^j} \right) dz^i \wedge d\bar z^j
\end{equation*}

This pseudo-K\"ahler structure is used in Section~\ref{instantonssec}
to study the moduli space of instantons (associative submanifolds)
of $M$. The two structures are clearly isomorphic, and it is easy
to see that they are related by the Legendre transform. See~\cite{L2}
for a similar construction in the context of mirror symmetry of semi-flat
Calabi-Yau manifolds.

We will also need the invariant descriptions of $\tilde \GJ$, $\tilde
\omega$, and $\tilde J$. At a point $(\ph, C)$ in $\mathcal M \times
\HtS \cong \mathcal J$, let $\eta_i \in T_{\ph} \mathcal M = \HtR$ and
$\mu_i \in T_{C} (\HtS) \cong \HtR$. It is easy to check that
\begin{equation} \label{standardsymplecticeq2}
\tilde \omega( (\eta_1, \mu_1), (\eta_2, \mu_2) ) =
\int_M \eta_1 \wedge \star_{\ph} \mu_2 - \int_M \eta_2 \wedge
\star_{\ph} \mu_1
\end{equation}
and
\begin{equation} \label{invariantmetriceq2}
\tilde \GJ( (\eta_1, \mu_1), (\eta_2, \mu_2) ) =
\int_M \eta_1 \wedge \star_{\ph} \eta_2 + \int_M \mu_1 \wedge
\star_{\ph} \mu_2
\end{equation}
and
\begin{equation} \label{invariantJeq2}
\tilde J(\eta, \mu) = (- \mu, \eta)
\end{equation}
These should be compared with~\eqref{standardsymplecticeq}, 
\eqref{invariantmetriceq}, and~\eqref{invariantJeq}, respectively.

We close this section with a discussion of the cubic form associated
to a Lagrangian fibration, and show that in this case it is precisely
the Yukawa coupling from Definition~\ref{yukawadefn}. From the work of
Donagi and Markman~\cite{DM2}, a Lagrangian fibration with section
corresponds to a certain cubic form on the base of the fibration.

The construction of the cubic form associated to the Lagrangian
fibration is as follows. By Remark~\ref{lagrmk}, we can identify
the tangent space $T_{\ph} \mathcal M$ to the base $\mathcal M$
with $\HfR$. For any tangent vector $\theta \in \HfR$ to the base,
we get an infinitesimal variation of the Lagrangian torus fibers
$\HfS$. Note that here we are fixing the lattice and changing the
metric on $\HfR$, instead of fixing the
metric on $\HfR$ and varying the lattice. Such an infinitesmal
variation is an element of $\mathrm{Sym}^2 (T^* \HfR) \cong
\mathrm{Sym}^2 (\HtR)$. We therefore have a map
\begin{equation*}
c : T\mathcal{M} \rightarrow \mathrm{Sym}^2 (\HtR)
\end{equation*}
or equivalently an element $c$ of 
\begin{equation*}
T^* \mathcal{M} \otimes \mathrm{Sym}^2 (\HtR) \cong \HtR \otimes
\mathrm{Sym} (\HtR)
\end{equation*}
We will show that the Lagrangian condition in fact implies that 
\begin{equation*}
c \in \mathrm{Sym}^3 (\HtR)
\end{equation*}
by showing that $c = 2 \, \mathcal{Y}$, where $\mathcal Y$ is the
Yukawa coupling. By the definition of $c$, if we take $\eta_1$,
$\eta_2$, and $\eta_3$ to be elements of $\HtR \cong T_{\ph}
\mathcal M$, then
\begin{equation*}
c\left( \eta_1, \eta_2, \eta_3 \right) = \left. \frac{d}{dt}
\right|_{t=0} \mathcal G_t \left( \eta_1, \eta_2 \right) = 
\left. \frac{d}{dt} \right|_{t=0} \int_M \eta_1 \wedge \star_{\ph_t}
\eta_2
\end{equation*}
where $\mathcal G_t$ is the Hessian metric of
Theorem~\ref{hessianmetricthm} at $T_{\ph_t} \mathcal M$, for
$\ph_t$ satisfying $\ph_0 = \ph$ and $\left. \frac{d}{dt}
\right|_{t=0} \ph_t = \eta_3$. But the claim now follows immediately
from Theorem~\ref{yukawathm}.

\section{Abel-Jacobi Maps in \G-geometry} \label{abeljacobisec}

In this section, we discuss Abel-Jacobi type maps in \G-geometry. We
show that associative cycles, coassociative cycles, and deformed
Donaldson-Thomas connections are critical points of Chern-Simons type
functionals. The moduli spaces of these structures can be
isotropically immersed (with respect to the appropriate symplectic
structure) into the universal intermediate Jacobian $\mathcal J$ using
these Abel-Jacobi maps.

\subsection{Functionals of Chern-Simons Type} \label{chernsimonssec}

We are going to define a Chern-Simons type functional for pairs
$\left( N, A \right) $ where $N$ is a $k$-dimensional submanifold of
$M$, of a fixed diffeomorphism type, and $A$ is a unitary connection
for a fixed rank $r$ complex vector bundle $E$ over $N$. Let $\mathcal
C_k$ be the set of all such pairs. We will fix a base point $\left(
N_0, A_0 \right) $ on each connected component of the configuration
space $\mathcal C_k$. For simplicity, we will restrict our attention to
one such connected component and denote its univeral cover by
$\widetilde {\mathcal C_k}$. An element in $\widetilde {\mathcal C_k}$
corresponds to a pair $\left( N, A \right)$ in $\mathcal C_k$ together
with a homotopy class of paths in $\widetilde {\mathcal C_k}$ joining
it with $\left( N_0, A_0 \right)$.  Let $\left\{ \left( N_t ,A_t
\right) \right\} _{t \in \left[ 0,1 \right]}$ be any such path.
We define the set
\begin{equation} \label{Nteq}
\bar N = \left\{ (t, p_t): t \in [0,1], p_t \in N_t \right\}
\end{equation}
Clearly we have $\bar N \, \cong \, [0,1] \times N_0$ since each $N_t$
is diffeomorphic to $N_0$. Let $\pi : \bar N \to M$ be the projection
onto the second factor: $\pi(t, p_t) = p_t \in N_t \subset M$. Then
$\pi^*(\ph)$ and $\pi^* (\st \ph)$ are forms on $\bar N$. From the family of
connections $A_t$ on $N_t$ we define
\begin{equation*}
\bar A(t,p_t) = \pi^*(A_t(p_t))
\end{equation*}
which is unitary connection $\bar A$ on $\bar N$.

\begin{defn} \label{chernsimonsdefn} We define the Chern-Simons type
functional $\Phi_{k, \ph}$ for fixed $k$ and fixed torsion-free
\G-structure $\ph$ as follows:

\begin{align} \nonumber
  \Phi_{k, \ph} & : \widetilde {\mathcal C_k} \rightarrow \mathbb R \\
  \label{chernsimonsdefneq} \Phi_{k, \ph} \left( N, A \right) & =
  \int_{\bar N} \tr \left[ \exp \left( \frac{i}{2\pi} F_{\bar A} +
      \pi^*(\ph) + \pi^*(\st \ph) \right) \right]
\end{align}
Here $F_{\bar A}$ is locally a matrix-valued differential form, and we
consider the ordinary forms $\ph$ and $\st \ph$ to be matrix-valued by
tensoring with the identity matrix. Hence $\tr(\alpha) = r\alpha$ for
an ordinary form $\alpha$, where $r$ is the rank of $E$.

By standard arguments using the closedness of $\ph$ and $\st \ph$ and
the Bianchi identity, this integral is independent of the choice of
the path in $\widetilde {\mathcal C_k}$.  Therefore, $\Phi_{k, \ph}$ is
a well-defined function. We will sometimes abuse notation and drop the
$\pi^*$ for notational simplicity, when there is no risk of confusion.
\end{defn}

\begin{rmk} \label{circlevaluedrmk}
If $\bar N$ is a {\em closed} manifold, then Stokes' theorem shows
that $\Phi_{k, \ph}$ is a well-defined functional on $\mathcal C_k$,
modulo discrete periods depending on the cohomology classes $[F_A]$,
$[\ph]$, and $[\st \ph]$. Therefore, up to a constant multiple, we
can assume that $\Phi_{k, \ph} (N, A)$ is always an integer for a
closed manifold $\bar N$, provided that $[\ph]$ and $[\st \ph]$ are
rational cohomology classes. In such case, we can descend $\Phi_{k,
\ph}$ to a {\em circle-valued function} on $\mathcal C_k$.
\end{rmk}

We will need the explicit form of this functional for paths such that
$N_t = N_0$ for all $t$. In this case $\bar N = [0,1] \times N_0$.
Then $A_t$ is a connection on $N_0$ for each $t$, and $t$ is just a
parameter. We have $F_{\bar A} = d\bar A + \bar A \wedge \bar A$,
where $d$ is the exterior derivative on $[0,1] \times N_0$. Hence
$F_{\bar A} = d_{N_0} A_t + dt \wedge A'_t + A_t \wedge A_t =
F_{A_t} + dt \wedge A'_t$, where $'$ denotes differentiation with
respect to the parameter $t$. Then it is easy to see that
\begin{equation*}
  \exp \left( \frac{i}{2\pi} F_{\bar A} + \ph + \st \ph \right) =
\left( 1 + \frac{i}{2\pi} dt \wedge A'_t \right) \wedge 
\exp \left( \frac{i}{2\pi} F_{A_t} + \ph + \st \ph \right)
\end{equation*}
The first term above has no $dt$ factor, so it integrates to zero over
$\bar N$. Therefore we are left with
\begin{equation} \label{Nfixedfunctionaleq}
\Phi_{k, \ph} (N_0, A) =
\frac{i}{2\pi} \int_0^1 \int_{N_0} dt \wedge \tr \left(  A'_t \wedge \exp
\left( \frac{i}{2\pi} F_{A_t} + \ph + \st \ph \right) \right)
\end{equation}
whenever $A_t$ is a path of connections from $A_0$ to $A$, for fixed $N_0$.

The symmetry group for this functional $\Phi_{k, \ph}$ is the
connected component of the group of extended gauge symmetries
$\widetilde {\mathcal G}$ of the bundle $E$ over $N_0$, which fits
into the following exact sequence:
\begin{equation*}
1 \rightarrow \mathcal G \rightarrow \widetilde {\mathcal G} \rightarrow
\mathrm{Diff}_0 ( N_0 ) \rightarrow 1
\end{equation*}
where $\mathcal G = \mathrm{Aut} (E)$ is the group of unitary gauge
transformations of $E$, and $\mathrm{Diff}_0 (N_0)$ is the group of diffeomorphisms of $N_0$ isotopic to the identity.

We denote the moduli space $\mathcal N_{k, \ph}$ of critical points of
$\Phi_{k, \ph}$ by
\begin{equation} \label{criticalmodulieq}
\mathcal N_{k, \ph} = \left\{  d\Phi_{k, \ph} = 0 \right\}  /
\widetilde {\mathcal G}
\end{equation}
and the associated {\em universal moduli space} by
\begin{equation} \label{universalmodulieq}
\mathcal N_k = \coprod_{\ph \in \mathcal M} \mathcal N_{k, \ph}
\end{equation}

Because we have $\exp \left( \ph + \st \ph \right) = 1 + \ph + \st \ph
+ \ph \wedge \st \ph$, we will see that the functional $\Phi_k$ is
only interesting for $k = 3$, $4$, or $7$. We are going to see that
the corresponding critical points are associative cycles,
coassociative cycles and deformed Donaldson-Thomas connections.

We now proceed to generalize the functionals $\Phi_{k, \ph}$ to obtain
analogues of {\em Abel-Jacobi maps} for \G-manifolds. Note that
since we are assuming that the cohomology of $M$ has no torsion, we
have
\begin{equation*}
(H^k(M, S^1))^* = H^{7-k}(M, S^1)
\end{equation*}
where recall that $H^k(M, S^1) = H^k(M, \mathbb R) /
H^k(M, \mathbb Z)$. Let $l = 3$ or $4$. We have shown that 
$\mathcal J \cong \mathcal M \times \HtS$ and also $\mathcal J \cong
\mathcal M \times \HfS$.

\begin{defn} \label{abeljacobidefn}
We define the Abel-Jacobi map $\nu^l_k$ from
the universal moduli space $\mathcal N_k$ to $\mathcal J \cong
\mathcal M \times H^{7-l}(M, S^1) \cong \mathcal M \times
(H^l(M, S^1))^*$ as follows.
\begin{eqnarray} \nonumber
\nu^l_k \quad : \quad \mathcal N_k & \to & \mathcal M \times
(H^l(M, S^1))^* \\ \label{abeljacobidefneq1} (\ph, N, A) &
\mapsto & (\ph, \beta)
\end{eqnarray}
where $\beta \in (H^l(M, S^1))^*$ is defined by
\begin{equation} \label{abeljacobidefneq2}
\int_M \beta \wedge \alpha = \int_{\bar N} \tr \left[ \exp
\left( \frac{i}{2\pi} F_{\bar A} + \pi^*(\ph) + \pi^*(\st \ph) \right)
\right] \wedge \pi^*(\alpha)
\end{equation}
for all $\alpha \in H^l(M, S^1)$.
\end{defn}

In the following sections, we will examine these maps in three
situations. For $k=3$ or $7$ we need $l=4$, and for $k=4$ we need
$l=3$. We will see explicitly that the Abel-Jacobi map $\nu^l_k$ is
well-defined in each of these cases, and we will show that its image
is isotropic with respect to the appropriate symplectic structure.
Hence $\mathcal N_k$ is a Lagrangian submanifold of $\mathcal J$
whenever $\nu$ is an immersion.

We close this section with the following lemma, which we will need
later.

\begin{lemma} \label{basiclemma}
Let $(M^n, g)$ be a compact oriented Riemannian manifold, with $N_0$ a
closed oriented $k$-dimensional submanifold of $M$. Suppose that $X$ is a
normal vector field on $N_0$. That is, $X(p) \perp T_p N_0$ for all
$p \in N_0$. Define
\begin{equation*}
N_s = \exp (s X) \cdot N_0 \qquad , \qquad \bar N(t) = 
 \left\{ (t, p_t): t \in [0,1], p_t \in N_t \right\}
\end{equation*}
and define $\pi : \bar N(t) \to M$ to be the projection onto the
second factor. We give $\bar N(t)$ the orientation such that $(
\frac{\partial}{\partial t}, e_1, \ldots, e_k)$ is an
oriented basis of $\bar N(t)$ whenever $(e_1, \ldots, e_k)$ is an oriented
basis of $N_s$. If $\alpha$ is a $(k+1)$-form on $M$, then
\begin{equation} \label{basiclemmaeq}
\left. \ddt \right|_{t=0} \int_{\bar N(t)} \pi^* (\alpha) =
\int_{N_0} X \hk \alpha
\end{equation}
\end{lemma}
\begin{proof}
Define a map $\rho : [0, t] \times N_0 \to \bar N(t)$ by $\rho(s,p) =
(s, \exp(sX) \cdot p)$. Clearly $\rho$ is a diffeomorphism, and $(\pi 
\circ \rho)_* (\frac{\partial}{\partial s}) = X$. We have
\begin{eqnarray*}
\int_{\bar N(t)} \pi^* (\alpha) & = & \int_{[0,t] \times N_0} \rho^* (
\pi^* (\alpha)) = \int_{[0,t] \times N_0} ds \wedge
\left(\frac{\partial}{\partial s} \hk (\pi \circ \rho)^* \alpha \right) \\
& = & \int_0^t \left( \int_{N_0} \frac{\partial}{\partial s} \hk
(\pi \circ \rho)^* \alpha \right) ds = \int_0^t
\left( \int_{N_0} X \hk \alpha \right) ds
\end{eqnarray*}
and thus by the fundamental theorem of calculus,
\begin{equation*}
\left. \ddt \right|_{t=0} \int_{\bar N(t)} \pi^*(\alpha) =
\int_{N_0} X \hk \alpha
\end{equation*}
which completes the proof.
\end{proof}

\subsection{Universal moduli of instantons}
\label{instantonssec}

An {\em instanton} in the \G-manifold $M^7$ is an {\em associative}
submanifold $N^3$. They are $3$-dimensional submanifolds of $M$
which are calibrated with respect to $\ph$. A $3$-manifold $N^3$ in $M$
is associative if and only if
\begin{equation} \label{assoceq}
\left. (X \hk \ps) \right|_N = 0
\end{equation}
for all normal vector fields $X$ on $N$, where $\ps = \st \ph$ is the
dual $4$-form.  This is equivalent to the vanishing on $N$ of the
vector valued $3$-form $\chi$ obtained by raising an index of $\ps$
with the metric $g_{\ph}$.

We will denote the functional $\Phi_{k, \ph}$ defined
in~\eqref{chernsimonsdefneq} by $\Phi_{\!\! \mathcal A}$ when $k=3$.
Explictly,
\begin{equation} \label{assocfunctionaleq}
\Phi_{\!\! \mathcal A} (N, A) = \Phi_{3, \ph} (N, A) = \int_{\bar N}
\left( r \, \pi^*(\ps) - \frac{1}{8 \pi^2} \tr (F_{\bar A}^2) \right)
\end{equation}
in this case. The critical points of $\Phi_{\!\! \mathcal A}$ are
described by the following theorem.

\begin{thm} \label{assocfunctionalthm}
The pair $(N_0, A_0)$ is a critical point of $\Phi_{\!\! \mathcal A}$ if and
only if $N_0$ is an associative submanifold of $M$ and $A_0$ is a {\em flat}
connection.
\end{thm}
\begin{proof}
We will choose $(N_0, A_0)$ to be our base point. Consider first a
variation of the connection: $A_t = A_0 + t A$, where $N = N_0$ is
fixed.  Hence $F_{A_t} = dA_t + A_t \wedge A_t = F_{A_0} + t (dA + A
\wedge A_0 + A_0 \wedge A) + t^2 (A \wedge A)$.

Then using equation~\eqref{Nfixedfunctionaleq}, (where $A'_t = A$) we have
\begin{eqnarray*}
\left. \ddt \right|_{t=0} \Phi_{\!\! \mathcal A} (N_0, A_t) & = &
\left. \ddt \right|_{t=0} -\frac{1}{4 \pi^2} \int_0^t \int_{N_0} ds \wedge
\tr \left( A \wedge F_{A_s} \right) \\ & = & -\frac{1}{4 \pi^2} \int_{N_0} 
\tr(A \wedge F_{A_0})
\end{eqnarray*}
This must vanish for all possible $A$, thus $F_{A_0} = 0$ and $A_0$ is
a flat connection.

Now let $N_t$ be a variation of $N_0$, which we can write as
$N_t = \exp(tX) \cdot N_0$, for some normal vector field $X$ on $N_0$.
By Lemma~\ref{basiclemma}, we have
\begin{equation*}
\left. \ddt \right|_{t=0} \Phi_{\!\! \mathcal A} (N_t, A_0) = r \int_{N_0}
\left( X \hk \ps \right) - 
\frac{1}{8 \pi^2} \int_{N_0} \left(  X \hk \tr (F_{A_0}^2) \right)
\end{equation*}
The second term vanishes since we just showed for $(N_0, A_0)$ a
critical point of $\Phi_{\!\! \mathcal A}$, we must have $F_{A_0} = 0$.
It follows from~\eqref{assoceq} that $(N_0, A_0)$ is critical
for all variations if and only if $N_0$ is associative and $A_0$ is flat.
\end{proof}

\begin{defn} \label{assocmodulidefn}
The pair $(N, A)$ where $N$ is an associative submanifold of $M$
with respect to $\ph$ and $A$ is a flat connection on $N$ will be
called an {\em associative cycle}. Using the notation
of~\eqref{criticalmodulieq} we see that $\mathcal N_{3, \ph}$
is the space of associative cycles with respect to the \G-structure
$\ph$.
\end{defn}

Now consider the Abel-Jacobi map from Definition~\eqref{abeljacobidefn}
with $k=3$. It is easy to see that this map is non-trivial only for
$l=4$. Let us denote $\nu^4_3$ by $\nu$. Explictly, we have
\begin{equation*}
\nu (\ph, N, A) = (\ph, \bar \nu (N, A)) = (\ph, C)
\end{equation*}
where $C = \bar \nu (N, A) \in \HtS$ is defined by
\begin{equation} \label{nudefneq}
\int_M \bar \nu (N, A) \wedge D = r \! \int_{\bar N} \pi^*(D)
\end{equation}
for any $D \in \HfS$, where $r$ is the rank of $E$.

\begin{prop} \label{nuprop}
The image of $\nu$ in $\mathcal J$ is isotropic with respect to the
symplectic structure $\tilde \omega$ on $\mathcal J \cong \mathcal M
\times \HtS$ described in~\eqref{standardsymplecticeq2}.
\end{prop}
\begin{proof}
The symplectic form is exact: $\tilde \omega = d\tilde \alpha$, where
$\tilde \alpha$ is the $1$-form given by
\begin{equation} \label{alphaeq2}
\tilde \alpha_{(\ph, C)} (\eta, \mu) = \frac{1}{2} \int_M \ph \wedge
\star_{\ph} \mu - \frac{1}{2} \int_M C \wedge \star_{\ph} \eta
\end{equation}
for $\eta \in T_{\ph} \mathcal M \cong \HtR$ and $\mu \in T_C \HtS \cong \HtR$.
(This follows easily from the local coordinate expression: 
$\tilde \omega = dx_j \wedge dy^j = d ( \frac{1}{2} x_j dy^j - \frac{1}{2}
y^j dx_j )$. We compute the pullback $\nu^* (\tilde \alpha)$ of $\tilde \alpha$
by $\nu$:
\begin{eqnarray} \nonumber
& & \left. \nu^* (\tilde \alpha) \right|_{(\ph, N, A)} (\delta \ph, \delta N, 
\delta A) = \left. \tilde \alpha \right|_{(\ph, \bar \nu (N, A))}
(\delta \ph, \bar \nu_* (\delta N, \delta A)) \\ \nonumber & & = 
\frac{1}{2} \int_M \ph \wedge \star_{\ph} \bar \nu_* (\delta N, \delta A) -
\frac{1}{2} \int_M \bar \nu (N, A) \wedge \star_{\ph} \delta \ph \\ 
\label{numaincalceq} & & = \frac{2}{3} \int_M \bar \nu_* (\delta N, \delta A) \wedge \ps -
\frac{1}{2} \int_M \bar \nu (N, A) \wedge \star_{\ph} \delta \ph
\end{eqnarray}
where we have used the fact that $\star_{\ph} \ph = \frac{4}{3} \stp \ph
= \frac{4}{3} \ps$. If we differentiate
equation~\eqref{nudefneq} with $D = \ps = \stp \ph$, and use
Lemma~\ref{basiclemma}, we obtain
\begin{equation*}
\int_M \bar \nu_* (\delta N, \delta A) \wedge \ps + \int_M \bar \nu (N, A) \wedge
\star_{\ph} \delta \ph = r \! \int_N (X \hk \ps) + r \! \int_{\bar N}
\pi^*( \star_{\ph} \delta \ph)
\end{equation*}
where $X$ is the normal vector field on $N$ corresponding to the variation
$\delta N$, and we have also used Lemma~\ref{ddtpslemma}.
Using~\eqref{nudefneq} again, the above expression reduces to
\begin{equation*}
\int_M \bar \nu_* (\delta N, \delta A) \wedge \ps = r \! \int_N (X \hk \ps)
\end{equation*}
Substituting this into~\eqref{numaincalceq} and
using~\eqref{nudefneq} again gives
\begin{equation*}
\left. \nu^* (\tilde \alpha) \right|_{(\ph, N, A)} (\delta \ph, \delta N, 
\delta A) = \frac{2r}{3} \int_N (X \hk \ps) - \frac{r}{2} \int_{\bar N} \pi^*(\star_{\ph} \delta \ph)
\end{equation*}
The first term vanishes since $N$ is associative. Now consider again the
functional $\Phi_{3, \ph}$, but consider also variations in $\ph$.
We see that
\begin{equation*}
\left. d\Phi_{3, \ph} \right|_{(\ph, N, A)} (\delta \ph, \delta N,
\delta A) = r \! \int_{\bar N} \pi^*(\star_{\ph} \delta \ph)
\end{equation*}
since the variations in $N$ and $A$ vanish because $(N,A)$ is an associative
cycle, and we have used Lemma~\ref{ddtpslemma} to compute the variation
in $\ph$. Comparing the two expressions, we see
\begin{equation*}
\nu^* (\tilde \alpha) = -\frac{1}{2} d \Phi_{3, \ph}
\end{equation*}
and hence $\nu^* (\tilde \omega) = \nu^*(d\tilde \alpha) = d
\nu^* (\tilde \alpha) = -\frac{1}{2} d^2 \Phi_{3, \ph} = 0$ as
claimed.
\end{proof}

\begin{rmk} \label{physics4rmk}
The proof of Proposition~\ref{nuprop} fails if one instead uses the function
$F = -\log(f)$ as the superpotential. This should be compared
to Remarks~\ref{physicsrmk}, ~\ref{physics2rmk}, and~\ref{physics3rmk}.
\end{rmk}

An interesting question concerns the non-degeneracy of the critical
points of $\Phi_{3, \ph}$. In~\cite{M} McLean proved that the local
deformations of associative submanifolds can be (and often are) obstructed.
It seems reasonable that our functional is non-degenerate at a critical point
$(N_0, A_0)$ if and only if the deformations are unobstructed at that point.
The dimension of the image $\nu(\mathcal N_3)$ of the Abel-Jacobi map
should also depend on the degeneracy of the critical points. One would like
to know exactly when this image is half-dimensional, so the image
is an honest Lagrangian submanifold.

The expected dimension of the moduli space of associative submanifolds
on a given \G-manifold is zero, since it is given by the index of a
twisted Dirac operator on an odd-dimensional manifold. If the kernel of
this Dirac operator is indeed trivial, then the moduli space is indeed
a discrete set. Furthermore, almost \G-manifolds also have a discrete
moduli space of associative submanifolds. When such generic situation
occurs, the image of our Abel-Jacobi map is half dimensional and
therefore a Lagrangian subspace in the universal intermediate
Jacobian. When the kernel is nontrivial, the moduli space of
asociative submanifolds could have positive dimension. However, at the
same time, such solutions may not be able to survive under
deformations of \Gs s, since the obstruction space is related
to the cokernel which is also nontrivial. Our Proposition~\ref{nuprop} can be
interpreted as a quantitative measure of such obstructions. In fact,
the image of $\nu$ may indeed be a Lagrangian subspace in general.
This situation should also be compared with the case of holomorphic curves,
which are the analogues of instantons for Calabi-Yau $3$-folds.

\subsection{Universal moduli of branes}
\label{branessec}

A {\em brane} in the \G-manifold $M^7$ is a {\em coassociative}
submanifold $L^4$. They are $4$-dimensional submanifolds of $M$
which are calibrated with respect to $\ps$. A $4$-manifold $L^4$ in $M$
is coassociative if and only if
\begin{equation} \label{coassoceq}
\left. \ph \right|_L = 0
\end{equation}
from which it is clear they are analogues of {\em Lagrangian} submanifolds
(see~\cite{LL2}.)

We will denote the functional $\Phi_{k, \ph}$ defined
in~\eqref{chernsimonsdefneq} by $\Phi_{\! \mathcal C}$ when $k=4$.
Explictly,
\begin{equation} \label{coassocfunctionaleq}
\Phi_{\! \mathcal C} (L, A) = \Phi_{4, \ph} (L, A) =
\frac{i}{\pi} \int_{\bar L} \tr \left( F_{\bar A} \wedge \pi^*(\ph) \right) 
\end{equation}
in this case. The critical points of $\Phi_{\! \mathcal C}$ are
described by the following theorem.

\begin{thm} \label{coassocfunctionalthm}
The pair $(L_0, A_0)$ is a critical point of $\Phi_{\! \mathcal C}$ if and
only if $L_0$ is a coassociative submanifold of $M$ and $\tr(F_{A_0})$ is a
{\em self-dual} $2$-form on $L_0$. In particular, when $E$ is a complex
line bundle (rank $r = 1$), then $F_{A_0}$ is self-dual, so $A_0$ is a
self-dual connection.
\end{thm}
\begin{proof}
We choose $(L_0, A_0)$ to be our base point. Consider first a
variation of the connection: $A_t = A_0 + t A$, where $L = L_0$ is
fixed. Then using equation~\eqref{Nfixedfunctionaleq},
(where $A'_t = A$) we have
\begin{eqnarray*}
\left. \ddt \right|_{t=0} \Phi_{\! \mathcal C} (L_0, A_t) & = &
\left. \ddt \right|_{t=0} \frac{i}{2 \pi} \int_0^t \int_{L_0} ds \wedge
\tr \left( A \wedge \ph \right) \\ & = & \frac{i}{2 \pi} \int_{L_0} 
\tr(A) \wedge \ph
\end{eqnarray*}
This must vanish for all possible $A$, hence $\left. \ph \right|_{L_0} = 0$
and thus by~\eqref{coassoceq}, $L_0$ is coassociative.

Now let $L_t$ be a variation of $L_0$, which we can write as
$L_t = \exp(tX) \cdot L_0$, for some normal vector field $X$ on $L_0$.
By Lemma~\ref{basiclemma}, we have
\begin{equation*}
\left. \ddt \right|_{t=0} \Phi_{\! \mathcal C} (L_t, A_0) = \frac{i}{\pi}
\int_{L_0} (X \hk \tr(F_{A_0})) \wedge \ph + \int_{L_0} \tr(F_{A_0})
\wedge (X \hk \ph)
\end{equation*}
The first term vanishes since we just showed for $(L_0, A_0)$ a
critical point of $\Phi_{\! \mathcal C}$, we must have $\ph = 0$ on
$L_0$. When $L_0$ is a coassociative submanifold, the map $X \mapsto
(X \hk \ph)$ gives an isomorphism between the normal vector fields $X$
on $L_0$ with the {\em anti-self-dual} $2$-forms on $L_0$.
(See~\cite{K2, M} for details.) Therefore since the above expression
must vanish for all normal vector fields $X$, we must have that
$\tr(F_{A_0})$ is orthogonal to $\Omega^2_-(L_0)$, so it is in
$\Omega^2_+(L_0)$. That is, $(L_0, A_0)$ is critical for all
variations if and only if $L_0$ is coassociative and $\tr(F_{A_0})$ is
a self-dual $2$-form.
\end{proof}

\begin{rmk} \label{signsrmk}
There are two different conventions for orientations in \G-geometry. With the
other convention, $(X \hk \ph)$ would be self-dual for all normal vector fields
$X$ on a coassociative, and $\tr(F_{A_0})$ would be anti-self-dual in the
above theorem. See~\cite{K2} for more about signs and orientations in
\G-geometry.
\end{rmk}

\begin{defn} \label{coassocmodulidefn}
The pair $(L, A)$ where $L$ is a coassociative submanifold of $M$
with respect to $\ph$ and $A$ is a connection on $N$ such that
$\tr(F_A)$ is self-dual will be called a {\em coassociative cycle}.
Using the notation of~\eqref{criticalmodulieq} we see that
$\mathcal N_{4, \ph}$ is the space of coassociative cycles with respect
to the \G-structure $\ph$.
\end{defn}

Now consider the Abel-Jacobi map from Definition~\eqref{abeljacobidefn}
with $k=4$. It is easy to see that this map is non-trivial only for
$l=3$. Let us denote $\nu^3_4$ by $\mu$. Explictly, we have
\begin{equation*}
\mu (\ph, L, A) = (\ph, \bar \mu (L, A)) = (\ph, D)
\end{equation*}
where $D = \bar \mu (L, A) \in \HfS$ is defined by
\begin{equation} \label{mudefneq}
\int_M \bar \mu (L, A) \wedge C = \frac{i}{2\pi} \int_{\bar L}
\tr(F_{\bar A}) \wedge C
\end{equation}
for any $C \in \HtS$.

\begin{prop} \label{muprop}
The image of $\mu$ in $\mathcal J$ is isotropic with respect to the
symplectic structure $\omega$ on $\mathcal J \cong \mathcal M \times
\HfS$ described in~\eqref{standardsymplecticeq}.
\end{prop}
\begin{proof}
Since the proof is very similar to Proposition~\ref{nuprop}, we will be brief.
This time $\omega = d \alpha$, where
\begin{equation} \label{alphaeq}
\alpha_{(\ph, D)} (\eta, \theta) = \frac{1}{2} \int_M \ph \wedge
\theta - \frac{1}{2} \int_M D \wedge \eta
\end{equation}
for $\eta \in T_{\ph} \mathcal M \cong \HtR$ and $\theta \in T_D \HfS \cong
\HfR$. Now one can compute as in the proof of Proposition~\ref{nuprop} that
\begin{multline*}
\left. \mu^* (\alpha) \right|_{(\ph, L, A)} (\delta \ph, \delta L, 
\delta A) = \frac{i}{4\pi} \int_L (X \hk \tr(F_A)) \wedge \ph \\ {}+ \frac{i}{4\pi} \int_L
\tr(F_A) \wedge (X \hk \ph) - \frac{i}{4\pi} \int_{\bar L} \tr(F_{\bar A}) \wedge \delta \ph
\end{multline*}
The first term vanishes since $L$ is coassociative and the second term vanishes
since in addition $\tr(F_A)$ is self-dual. Now consider again the
functional $\Phi_{4, \ph}$, but consider also variations in $\ph$.
We have
\begin{equation*}
\left. d\Phi_{4, \ph} \right|_{(\ph, L, A)} (\delta \ph, \delta L,
\delta A) = \frac{i}{2\pi} \int_{\bar L} \tr(F_{\bar A}) \wedge \delta \ph
\end{equation*}
since the variations in $L$ and $A$ vanish because $(L,A)$ is a coassociative
cycle. Thus
\begin{equation*}
\mu^* (\alpha) = -\frac{1}{2} d \Phi_{4, \ph}
\end{equation*}
and hence as before $\mu^*(\omega) = 0$.
\end{proof}

As in the associative case, it would be interesting to be able to
relate the non-degeneracy of the critical points of $\Phi_{4, \ph}$ to
the deformation theory. In constrast to the instanton case,
in~\cite{M} McLean proved that the local deformations of coassociative
submanifolds are always unobstructed. The moduli space of
coassociative submanifolds is smooth and of dimension $b^2_-(L)$. On the
other hand, a generic $L$ does not support any self-dual connections of
rank $1$. They occur discretely on generic $b^2_-(L)$-dimensional families of
metrics on $L$. Therefore, when such genericity occurs, then our above
arguments really show that the image of Abel-Jacobi map is indeed a
Lagrangian subspace in the universal intermediate Jacobian.

\subsection{Universal moduli of Donaldson-Thomas connections}
\label{DTsec}

Finally, we consider the functional $\Phi_{k, \ph}$
from~\eqref{chernsimonsdefneq} for $k=7$. In this case $N = M$ necessarily, and
thus we can only vary a connection $A$ on a bundle $E$ over $M$. We will
denote the functional by $\Phi_{\! \mathcal{DT}}$ this time.
Explictly,
\begin{equation*}
\Phi_{\! \mathcal{DT}} (A) = \Phi_{7, \ph} (M, A) =
\int_{\bar M} \tr \left( \frac{3}{6} {\left( \frac{i}{2\pi} F_{\bar A}
\right)}^2 \wedge \pi^*(\ps) + \frac{1}{24} {\left( \frac{i}{2\pi} F_{\bar A}
\right)}^4 \right)
\end{equation*}
where $\bar M = [0,1] \times M$ and $\bar A = \pi^*(A_t)$.
We also have that
\begin{equation} \label{DTfunctionaleq}
\Phi_{\! \mathcal{DT}} (A) = -\frac{1}{4 \pi^2} \int_0^1 \int_M dt \wedge 
\tr \left( A'_t \wedge \left( F_{A_t} \wedge \ps - \frac{1}{24 \pi^2}
F_{A_t}^3 \right) \right)
\end{equation}
by using~\eqref{Nfixedfunctionaleq}. The critical points of $\Phi_{\!
  \mathcal DT}$ are described by the following theorem.

\begin{thm} \label{DTfunctionalthm}
The connection $A_0$ is a critical point of $\Phi_{\! \mathcal DT}$ if and
only if $A_0$ satisfies the {\em deformed Donaldson-Thomas} equation:
\begin{equation} \label{DTeq}
F_{A_0} \wedge \ps - \frac{1}{24 \pi^2} F_{A_0}^3 = 0
\end{equation}
\end{thm}
\begin{proof}
We will choose $A_0$ to be our base point. Consider a variation of
the connection: $A_t = A_0 + t A$. Hence $F_{A_t} = F_{A_0} + t (\cdots) + t^2
(\cdots)$. Using equation~\eqref{DTeq}, (where $A'_t = A$) we have
$\left. \ddt \right|_{t=0} \Phi_{\! \mathcal DT} (A_t) =$
\begin{eqnarray*}
& & \left. \ddt \right|_{t=0} -\frac{1}{4 \pi^2} \int_0^t \int_{M} ds \wedge
\tr \left( A \wedge \left( F_{A_0} \wedge \ps - \frac{1}{24 \pi^2}
F_{A_0}^3 + s (\cdots) \right) \right) \\
& = & -\frac{1}{4 \pi^2} \int_{M} \tr \left( A \wedge \left( F_{A_0}
\wedge \ps - \frac{1}{24 \pi^2} F_{A_0}^3 \right) \right)
\end{eqnarray*}
This must vanish for all possible $A$, and thus the theorem follows.
\end{proof}

\begin{rmk} \label{DTrmk}
In~\cite{DT}, Donaldson and Thomas introduced the equation
\begin{equation*}
F_A \wedge \ps = 0
\end{equation*}
for a connection $A$ on a \G-manifold $M$, as a generalization of the
{\em anti-self-dual} equations of Yang-Mills theory, to this setting.
Such connections have curvature $2$-form $F_A$ contained in $\wtwf
\cong \lieg$.  The deformed Donaldson-Thomas equation was introduced
in~\cite{LL}, where $F_A$ is used instead of $\frac{i}{2\pi} F_A$. The
cubic correction term allows the deformed DT equation to fit into the
general setting described in Section~\ref{chernsimonssec}. This
`deformed' equation has been studied by physicists
in~\cite{ddES1}.

One can obtain the original Donaldson-Thomas connections from the
same functional after truncating the Chern character terms. That is,
replacing $\exp$ by $\exp_{\leq2} ( x ) = 1 + x + \frac{1}{2} x^2$.
\end{rmk}

\begin{defn} \label{DTmodulidefn}
A connection $A$ satisfying equation~\eqref{DTeq} will be called
a {\em DDT connection}. Using the notation
of~\eqref{criticalmodulieq} we see that $\mathcal N_{7, \ph}$
is the space of DDT connections with respect to the \G-structure
$\ph$.
\end{defn}

Now consider the Abel-Jacobi map from Definition~\eqref{abeljacobidefn}
with $k=7$, and take $l=4$. Let us denote $\nu^4_7$ by $\chi$.
Explictly, we have
\begin{equation*}
\chi (\ph, A) = (\ph, \bar \chi (A)) = (\ph, C)
\end{equation*}
where $C = \bar \chi (A) \in \HtS$ is defined by
\begin{equation} \label{chidefneq}
\int_M \bar \chi (A) \wedge D = \int_{M \times [0,1]}
\left( r \, \pi^*(\st \ph) - \frac{1}{8\pi^2} \tr(F_{A_t}^2) \right)
\wedge \pi^*(D)
\end{equation}
for any $D \in \HfS$, where $r$ is the rank of $E$. Notice that the
first term has no $dt$ component, and thus must vanish on $\bar M = M
\times [0,1]$. Then since $\frac{i}{2\pi} \tr(F_{A_t})$ is an integral
class, this map is well-defined.

\begin{prop} \label{chiprop}
The image of $\chi$ in $\mathcal J$ is isotropic with respect to the
symplectic structure $\tilde \omega$ on $\mathcal J \cong \mathcal M
\times \HtS$ described in~\eqref{standardsymplecticeq2}.
\end{prop}
\begin{proof}
Proceeding as in the proof of Proposition~\ref{nuprop}, it is
easy to check that
\begin{equation*}
\chi^* (\tilde \alpha) = d \Phi_{7, \ph}
\end{equation*}
where $\tilde \omega = d \tilde \alpha$ is the symplectic form.
One needs to use the fact that $(X \hk \ps) \wedge
\ps$ vanishes for any vector field $X$, and that an $8$-form on $\bar M
= M \times [0,1]$ vanishes unless it has a $dt$ factor. The details are
left to the reader.
\end{proof}

As before, the non-degeneracy of the critical points of $\Phi_{7,
\ph}$ may be related to the deformation theory of deformed
Donaldson-Thomas connections. As far as we know this has not been
studied.

\subsection{Universal Moduli as Minimizers} \label{minimizerssec}

Now we consider another functional $\Psi$ which looks very similar 
to $\Phi,$ but is slightly different. It is defined as follows.

\begin{eqnarray*}
\Psi_{k, \ph} & : & \widetilde{\mathcal C_k} \rightarrow \mathbb R \\
\Psi_{k, \ph} \left( N, A \right) & = & \int_{N} \tr \left[ \exp \left(
\frac{i}{2\pi}F_{A} + \ph + \st \ph \right) \right]
\end{eqnarray*}
We also write $\Psi_{k} \left( \ph, N, A \right) = \Psi_{k, \ph}
\left( N, A \right)$. Note that since each $N$ in our configuration
space is in the same homology class, $\Psi_{k, \ph}$ is a {\em
topological number}. Let $\mathcal{YM} (A)$ denote
the Yang-Mills functional of the connection $A$. The {\em size}
of a pair $(N,A)$ in $\widetilde{\mathcal C_k}$ is defined to be
$\vol(N) + \mathcal{YM} (A)$.

Suppose that $k=3$. Then we have
\begin{equation*}
\Psi_{3, \ph} \left( N, A \right) = \int_N \ph \, \leq \, \vol \left( N \right)
+ \mathcal{YM} \left( A \right)
\end{equation*}
and the minimum is attained exactly along the critical points of 
$\Phi_{3, \ph}$, namely associative submanifolds coupled with flat
connections.

When $k = 4,$ we need to assume that $\mathrm{c}_2(E) < 0$, so that
the absolute minima of $\mathcal{YM}(A)$ are the self-dual connections. Then
\begin{equation*}
\Psi_{4, \ph} \left( L, A \right) = \int_L \left( \st \ph -
\frac{1}{8\pi^2} \tr \left( F_A^{2} \right) \right) \, \leq \, \vol
\left( L \right) + \mathcal{YM} \left( A \right)
\end{equation*}
and the minimum is attained along the coassociative submanifolds 
coupled with self-dual connections over them.

Finally when $k = 7$, that is when $N = M$, we have
\begin{equation*}
\Psi_{7, \ph} \left( M, A \right) = \int_M \left( \ph \wedge \st \ph
- \frac{1}{8 \pi^2} \tr \left( F_A^2 \right) \wedge \ph \right) \, \leq \, 
\vol \left( M \right) + \mathcal{YM} \left( A \right)
\end{equation*}
and here the minimum is attained for Donaldson-Thomas
connections over $M$, as explained in the section on Yang-Mills
calibrations in~\cite{L3}.

Therefore we see that this topological number $\Psi_{k, \ph}$ in each
case serves as an absolute lower bound for the size of $(N,A)$, and
this lower bound is attained exactly at the critical points of the
functional $\Phi_{k,\ph}$. (This statement requires the rank $r=1$ for
the second case, and we are using ordinary Donaldson-Thomas
connections here instead of DDT connections for the third case.)

In this last case, if we also allow the $3$-form $\ph$ to vary in a
fixed cohomology class of $M$, then the first term is the Hitchin
functional defined in~\cite{Hi4}, and the critical points, which are
also minimum points, are precisely given by torsion free \G-structures
on $M$. Quantizing the Hitchin functional has been discussed in many
physics papers, including~\cite{AW} and~\cite{ddES2}. Therefore,
it is reasonable to expect that one could obtain interesting physics
by quantizing our functional as well.

\section{Conclusion} \label{conclusionsec}

We close by considering generalizations of our constructions as well as
questions for future studies.

\subsection{Generalizations} \label{generalizationssec}

In this paper, we have shown that the universal intermediate 
Jacobian $\mathcal J$ of a \G-manifold has a natural pseudo-K\"ahler
structure and it admits many Lagrangian subspaces given by the universal
moduli of associative cycles, coassociative cycles, and deformed
Donaldson-Thomas connections. All these are critical points for our
functional $\Phi$.

For any given torsion-free \G-structure $\ph$, its
intermediate Jacobian $\mathcal J_{\ph}$ is also a Lagrangian 
submanifold of $\mathcal J$ and together they form a Lagrangian fibration over 
$\mathcal M$. We can also replace $\mathcal J_{\ph}$ by the moduli space of
flat $2$-gerbes on $M$ and obtain similar pictures. A good reference for gerbes, from a differential-geometric point of view, is Hitchin~\cite{Hi3}.

In fact all of these Lagrangian submanifolds $\mathcal J_{\ph} \subset
\mathcal J$ can also be described as critical points of our functional
$\Phi$, if we replace $F_{\bar A}$ by the curvature of a gerbe as
follows. Let $E$ be a $\U$ $2$-gerbe on a manifold 
$\bar N$. A $2$-connection on $E$ is locally given by a $3$-form $C$ on $M$
and its $2$-curvature $F_C = dC$ is a well-defined closed $4$-form on
$\bar N$. The cohomology class $c (E) = \left[ \frac{i}{2\pi} F_C \right] \in
H^4 ( \bar N , \mathbb R )$ is independent of the choice of $C$, 
which is analogous to the first Chern class of a complex line bundle.

Consider $k=7$, and let us fix a $\U$ $2$-gerbe $E$ on $M$ together with
a fixed $2$-gerbe connection $C_0$ on it. We define a functional on the 
space of $2$-gerbe connections on $E$, and we denote this configuration space 
by $\mathcal C$ again.
\begin{eqnarray*}
\Phi_{7, \ph} & : & \widetilde{\mathcal C} \rightarrow \mathbb R \\
\Phi_{7, \ph} \left( A \right) & = & \int_{\bar N} \exp \left(  
\frac{i}{2\pi} F_{\bar C} + \ph + \st \ph \right) \\
& = & \frac{1}{2} \int_{\bar N} \left( \frac{i}{2\pi} F_{\bar C} + \st \ph
\right)^2
\end{eqnarray*}
where $\bar N = M \times [ 0, 1]$ and the $2$-connection 
$\bar C$ over it is determined by a path of $2$-connections $C_t$ over $M$ 
joining $C_0$ and $C$.

As before, the Euler-Lagrangian equation for $\Phi_{7, \ph}$ is 
given by
\begin{equation*}
\frac{i}{2\pi} F_{\bar C} + \st \ph = 0
\end{equation*}
Hence the moduli space of solutions $\mathcal N_{7, \ph}$ is 
non-empty if and only if $c (E) = - [ \st \ph ] \in
\HfR$, in which case $\mathcal N_{7, \ph}$ is
isomorphic to the space of flat $2$-gerbes over $M$, which is $\HtS$.
Furthermore this equals the union of moduli $\mathcal N_7$.
Thus this space $\mathcal N_7$ coincides with the Lagrangian fibers of the 
universal intermediate Jacobian $\mathcal J \rightarrow \mathcal M$.

\subsection{Questions for future study} \label{questionssec}

It would be interesting to see explicitly how these constructions
relate to the analgous ones in the case of a Calabi-Yau $3$-fold. Recall
that the universal intermediate Jacobian $\mathcal J$ in the Calabi-Yau $3$-fold case
is {\em hyperK\"ahler}, thus it admits a whole $S^2$ family of
complex structures. In the \G\ case these reduce to just one complex structure.
While it is true that $M^7 = X^6 \times S^1$ admits a torsion-free \G-structure
when $X^6$ is a Calabi-Yau $3$-fold, the precise relationship between
the moduli spaces of these structures is more complicated. This is because
the complex and K\"ahler moduli of $X^6$ mix together in a non-trivial way
to yield the \G-moduli of $M^7$. This is reflected in the fact that in this
case, $b_1(M) = 1$, and most of the results obtained in this paper were
specific to the case $b_1(M) = 0$, which corresponds to full holonomy \G.

Two more difficult questions which need to be pursued are the following:
\begin{itemize}
\item What is the intersection theory of all these Lagrangian subspaces in
the universal intermediate \G-Jacobian $\mathcal J$?
\item What is the relation of these constructions with the
{\em topological quantum field theory} (TQFT) of Calabi-Yau $3$-folds
and \G-manifolds, as described, for example, in~\cite{L1}?
\end{itemize}
We hope to address these questions in future work.

Another situation that should be studied is when the \G-manifold is
non-compact, with an asymptotically cylindrical end, asymptotic to a
cylinder on a Calabi-Yau $3$-fold. Then the corresponding
constructions on the $\mathrm{CY}$ $3$-fold give the hyperK\"ahler
universal intermediate Jacobian, and various complex Lagrangian submanifolds inside
of it. (See~\cite{Ko1, Ko2} for more about asymptotically cylindrical Calabi-Yau $3$-folds and $\G$ manifolds.) The relation between the two constructions in this case needs
to be understood.  In addition the Topological Field Theory picture of
such a \G/$\mathrm{CY}3$ situation should be part of Topological
M-Theory. Similarly one could consider a \G-manifold which is the end of a
\SP-manifold with an asymptotically cylindrical end. Then the \G/\SP\ story
should be an analogue/generalization of the Chern-Simons/Donaldson Topological
Field Theory story.

There also remains much to understand about the differential geometry of
the moduli space $\mathcal M$ of torsion-free \G-structures. In particular,
the Hitchin metric $\mathcal G$ deserves more attention.
Theorem~\ref{yukawathm}, our main result relating the Yukawa coupling
to the superpotential function $f$, is used in~\cite{KL}
to study the {\em curvature} of this metric.

\end{document}